\documentclass[reqno]{amsart}
\usepackage{amsfonts}

\newtheorem{theorem}{Theorem}
\theoremstyle{plain}
\newtheorem{acknowledgement}{Acknowledgement}

\newtheorem{corollary}{Corollary}

\newtheorem{definition}{Definition}

\newtheorem{lemma}{Lemma}

\newtheorem{proposition}{Proposition}
\newtheorem{remark}{Remark}

\DeclareMathOperator{\Div}{div}
 \numberwithin{equation}{section}
\input{tcilatex}

\begin{document}
\title[Homogenization of reaction-diffusion equation]{Periodic
Homogenization of strongly nonlinear reaction-diffusion equations with large
reaction terms}
\author{Nils Svanstedt}
\address{N. Svanstedt, Department of Mathematical Sciences, G\"{o}teborg
University, SE-412 96 G\"{o}teborg, Sweden}
\email{nilss@math.chalmers.se}
\author{Jean Louis Woukeng}
\address{Jean Louis Woukeng, Department of Mathematics and Computer Science,
University of Dschang, P.O. Box 67, Dschang, Cameroon}
\curraddr{Jean Louis Woukeng, Department of Mathematics and Applied
Mathematics, University of Pretoria, Pretoria 0002, South Africa}
\email{jwoukeng@yahoo.fr}
\date{October, 2011}
\subjclass[2000]{35B27, 76M50}
\keywords{Homogenization, reaction-diffusion, nonlinear}

\begin{abstract}
We study in this paper the periodic homogenization problem related to a
strongly nonlinear reaction-diffusion equation. Owing to the large reaction
term, the homogenized equation has a rather quite different form \ which
puts together both the reaction and convection effects. We show in a special
case that, the homogenized equation is exactly of a convection-diffusion
type. The study relies on a suitable version of the well-known two-scale
convergence method.
\end{abstract}

\maketitle

\section{Introduction}

The aim of this work is the study of the asymptotic behavior of the
solutions of an initial boundary value problem for a strongly nonlinear
reaction-diffusion equation with a large reaction term, in a cylinder. The
equation reads as:
\begin{eqnarray}
\rho \left( \frac{x}{\varepsilon }\right) \frac{\partial u_{\varepsilon }}{%
\partial t} &=&\Div a\left( x,t,\frac{x}{\varepsilon },\frac{t}{\varepsilon
^{k}},Du_{\varepsilon }\right) +\frac{1}{\varepsilon }g\left( \frac{x}{%
\varepsilon },\frac{t}{\varepsilon ^{k}},u_{\varepsilon }\right) \text{\ in }%
Q_{T}  \label{1.1} \\
u_{\varepsilon } &=&0\text{\ on }\partial Q\times \left( 0,T\right)  \notag
\\
u_{\varepsilon }(x,0) &=&u^{0}(x)\in L^{2}(Q)  \notag
\end{eqnarray}%
where $Q_{T}=Q\times \left( 0,T\right) $ is our cylinder and $k$ is a given
positive parameter. The motivation of this study comes essentially from the
applicability of the preceding model. In fact, when the function $%
a(x,t,y,\tau ,\lambda )$ is linear with respect to $\lambda $, that is, $%
a(x,t,y,\tau ,\lambda )=b(x,t,y,\tau )\cdot \lambda $, the unknown $%
u_{\varepsilon }$ may be viewed as the concentration of some chemical
products diffusing in a porous medium of porosity $\rho (y)$, with varying
diffusivity $b(x,t,y,\tau )$ and reacting with the background medium by
absorption/desorption through the term $g(y,\tau ,r)$ \cite{AllPiat1}. The
fact that the diffusivity depends on the macroscopic variables $(x,t)$ means
that the concentration varies locally (and not uniformly) in the medium.
When the diffusivity is nonlinear as it is the case in (\ref{1.1}) and has
the specific form $a(x,t,y,\tau ,\lambda )=b(x,t,y,\tau )\left\vert \lambda
\right\vert ^{p-2}\lambda $, we obtain some model equations of porous media
\cite{Amaziane} (see also \cite{Antontsev}); here $u_{\varepsilon }$ is the
density of the fluid, $\rho (y)$ and $b(x,t,y,\tau )$ are respectively the
porosity and the permeability of the medium.\medskip

To proceed with the study of our model, we apply general ideas of
homogenization \cite{BLP, Jikov} and specifically the framework of two-scale
convergence introduced in \cite{Nguetseng} and developed in \cite{Allaire}.
Although the homogenization process is standard, it has still some
difficulties in our situation. In fact, the diffusion term is nonlinear, and
the lower order term $\frac{1}{\varepsilon }g(x/\varepsilon ,t/\varepsilon
^{k},u_{\varepsilon })$ is large because of the presence of the factor $%
1/\varepsilon $. To avoid obtaining a resulting homogenized equation of
stochastic's type, we assume a centering type condition on the function $g$,
that is the periodic function $(y,\tau )\mapsto g(y,\tau ,r)$ has zero mean
value with respect to the variable $y$, which then allows us to express $g$
as the gradient of a regular function. We also use this condition in both
the a priori estimates and the passage to the limit. This produces a limit
problem of a completely different type, which puts together both the
reaction and convection effects; see Proposition \ref{p5.3}. To be more
precise, here is the main result of the paper (the assumptions are to be
specified later).

\begin{theorem}
Let $2\leq p<\infty $. Assume hypotheses \textbf{A1}-\textbf{A5} hold. For
each $\varepsilon >0$ let $u_{\varepsilon }$ be the unique solution to \emph{%
(\ref{1.1})}. Then there exists a subsequence of $\varepsilon $ not
relabeled such that $u_{\varepsilon }\rightarrow u_{0}$ in $L^{2}(Q_{T})$
where $u_{0}\in L^{p}(0,T;W_{0}^{1,p}(Q))$ is solution to the following
boundary value problem:
\begin{equation}
\left\{
\begin{array}{l}
\frac{\partial u_{0}}{\partial t}=\Div q(\cdot ,\cdot
,u_{0},Du_{0})+q_{0}(\cdot ,\cdot ,u_{0},Du_{0})\text{\ in }Q_{T} \\
u_{0}=0\text{\ on }\partial Q\times (0,T) \\
u_{0}(x,0)=u^{0}(x)\text{\ in }Q\text{.}%
\end{array}%
\right.  \label{1.2}
\end{equation}
\end{theorem}

The main issue in getting (\ref{1.2}) lies at the level that the derivative
with respect to time $\partial u_{\varepsilon }/\partial t$ involves a
weight function represented by $\rho (x/\varepsilon )$. Indeed, with the
presence of $\rho (x/\varepsilon )$ the usual Aubin-Lions compactness result
\cite[Chap. 1, p. 58]{Lions} does not apply to our situation, and we use an
appropriate one due to Amar et al. \cite[Theorem 2.3]{Amar} and generalizing
the former. Also, because of $\rho (x/\varepsilon )$, the space of test
functions in the homogenization process is strongly modified. In the
framework of the usual two-scale convergence, the test functions are usually
taken in a space of the type $\mathcal{C}_{\text{per}}^{\infty }(Y)$. Here,
because of the function $\rho $, this space is reduced to those functions $u$
in $\mathcal{C}_{\text{per}}^{\infty }(Y)$ satisfying the additional \textit{%
normalized condition} $\int_{Y}\rho (y)u(y)dy=0$. This condition plays a
crucial role in the choice of the correction term $u_{1}$, which must then
satisfy the same assumption itself. Another consequence of this choice is
that one must prove the density of the space $\{u\in \mathcal{C}_{\text{per}%
}^{\infty }(Y):\int_{Y}\rho (y)u(y)dy=0\}$ in the space $\{u\in W_{\text{per}%
}^{1,p}(Y):\int_{Y}\rho (y)u(y)dy=0\}$. Also, due to the form of the
homogenized problem (which might be degenerate) there is no general
uniqueness result for the homogenized equation (\ref{1.2}). However, we show
that in some cases, there is uniqueness of the solution to the said problem.

\medskip

There is a variety of papers dealing with homogenization of operators of the
same type as (\ref{1.1}) but with a linear diffusion term which is not
depending on the macroscopic variables $(x,t)$. Without any pretension of
exhaustiveness we refer to \cite{AllPiat1} (for the case when $k=2$), to
\cite{PardouxPiat} (in which $\rho \equiv 1$ and $k=2$, but the behavior in
the microscopic time variable being with respect to some ergodic diffusion
process $\xi _{t/\varepsilon ^{2}}$) and to \cite{DiopPardoux} in which the
following operator is considered:
\begin{equation*}
\frac{\partial u_{\varepsilon }}{\partial t}-\Div\left( a\left( \frac{x}{%
\varepsilon },\frac{t}{\varepsilon ^{k}}\right) Du_{\varepsilon }\right) +%
\frac{1}{\varepsilon ^{\max (1,k/2)}}g\left( \frac{x}{\varepsilon },\frac{t}{%
\varepsilon ^{k}},u_{\varepsilon }\right) +h\left( \frac{x}{\varepsilon },%
\frac{t}{\varepsilon ^{k}},u_{\varepsilon }\right)
\end{equation*}%
with the same assumptions as in \cite{PardouxPiat}.\medskip

The paper is organized as follows. In Section 2, we recall the concept of
two-scale convergence. We adapt it to the situation of the problem (\ref{1.1}%
). Section 3 deals with a priori estimates of the solution of the problem (%
\ref{1.1}). In Section 4, we give some preliminary results that will be used
in the next section. Finally, Section 5 deals with the homogenization
results for (\ref{1.1}). We also study there a particular case when the
homogenized problem possesses a unique solution, and show that the whole
sequence converges in that case to the solution of a problem of
convection-diffusion type.\medskip

We end this section with some notations. All functions are assumed real
valued and all function spaces are considered over $\mathbb{R}$. Let $%
Y=\left( 0,1\right) ^{N}$ and let $F(\mathbb{R}^{N})$ be a given function
space. We denote by $F_{\text{per}}(Y)$ the space of functions in $F_{\text{%
loc}}(\mathbb{R}^{N})$ (when it makes sense) that are $Y$-periodic. Given a $%
Y$-periodic function $\rho $, we denote by $F_{\#\rho }(Y)$ the subspace of $%
F_{\text{per}}(Y)$ consisting of functions $u$ for which $\rho u$ has mean
value zero: $\int_{Y}\rho (y)u(y)dy=0$. As special cases, $\mathcal{D}_{%
\text{per}}(Y)$ denotes the space $\mathcal{C}_{\text{per}}^{\infty }(Y)$
while $\mathcal{D}_{\#\rho }(Y)$ stands for the space of those functions $u$
in $\mathcal{D}_{\text{per}}(Y)$ for which $\rho u$ has mean value zero. $%
\mathcal{D}_{\text{per}}^{\prime }(Y)$ stands for the topological dual of $%
\mathcal{D}_{\text{per}}(Y)$ which can be identified to the space of
periodic distributions in $\mathcal{D}^{\prime }(\mathbb{R}^{N})$.

\section{Two-scale convergence}

We recall the notion of two-scale convergence \cite{Allaire, Nguetseng}. We
adapt it to our framework and get the following

\begin{definition}
\label{d2.1}\emph{A sequence }$(u_{\varepsilon })_{\varepsilon >0}\subset
L^{p}(Q_{T})$\emph{\ (}$1\leq p<\infty $\emph{) is said to two-scale
converge towards }$u_{0}\in L^{p}(Q_{T}\times Y\times Z)$\emph{\ (}$Z=(0,1)$%
\emph{) if, as }$\varepsilon \rightarrow 0$\emph{, }%
\begin{equation}
\int_{Q_{T}}u_{\varepsilon }(x,t)f\left( x,t,\frac{x}{\varepsilon },\frac{t}{%
\varepsilon ^{k}}\right) dxdt\rightarrow \iint_{Q_{T}\times Y\times
Z}u_{0}(x,t,y,\tau )f(x,t,y,\tau )dxdtdyd\tau  \label{2.1}
\end{equation}%
\emph{for all }$f\in L^{p^{\prime }}(Q_{T};\mathcal{C}_{\text{\emph{per}}%
}(Y\times Z))$\emph{. We denote it by }$u_{\varepsilon }\rightarrow u_{0}$%
\emph{\ in }$L^{p}(Q_{T})$\emph{-2s.}
\end{definition}

The following two compactness results are well-known in the literature; see
e.g. \cite{NgWou} for the exact situation considered here.

\begin{theorem}
\label{t2.1}Let $1<p<\infty $. Then any bounded sequence in $L^{p}(Q_{T})$
admits a two-scale convergent subsequence.
\end{theorem}

\begin{theorem}
\label{t2.2}Let $1<p<\infty $. Let $(u_{\varepsilon })_{\varepsilon \in E}$
(where $E$ is an ordinary sequence of real numbers converging to zero with $%
\varepsilon $) be a bounded sequence in $L^{p}(0,T;W_{0}^{1,p}(Q))$. There
exist a subsequence of $E$ denoted by $E^{\prime }$, and a couple $%
(u_{0},u_{1})\in L^{p}(0,T;W_{0}^{1,p}(Q))\times L^{p}(Q_{T}\times Z;W_{%
\text{\emph{per}}}^{1,p}(Y))$ such that, as $E^{\prime }\ni \varepsilon
\rightarrow 0$,
\begin{equation*}
u_{\varepsilon }\rightarrow u_{0}\text{ in }L^{p}(0,T;W_{0}^{1,p}(Q))\text{%
-weak}
\end{equation*}%
and
\begin{equation*}
\frac{\partial u_{\varepsilon }}{\partial x_{j}}\rightarrow \frac{\partial
u_{0}}{\partial x_{j}}+\frac{\partial u_{1}}{\partial y_{j}}\text{ in }%
L^{p}(Q_{T})\text{-2s }(1\leq j\leq N).
\end{equation*}
\end{theorem}

In Theorem \ref{t2.2} the function $u_{1}$ is unique up to an additive
function of variables $x,t,\tau $. We need to fix its choice in accordance
with the needs in the sequel. For that, let us recall the definition of the
space $W_{\#\rho }^{1,p}(Y)$ for a given positive function $\rho \in L_{%
\text{per}}^{\infty }(Y)$ with non zero mean value:
\begin{equation*}
W_{\#\rho }^{1,p}(Y)=\left\{ u\in W_{\text{per}}^{1,p}(Y):\int_{Y}\rho
(y)u(y)dy=0\right\} .
\end{equation*}%
$W_{\#\rho }^{1,p}(Y)$ is a closed subspace of $W_{\text{per}}^{1,p}(Y)$
since it is the kernel of the continuous linear functional $u\mapsto
\int_{Y}\rho (y)u(y)dy$ defined on $W_{\text{per}}^{1,p}(Y)$. The following
version of Theorem \ref{t2.2} will be used in the sequel.

\begin{theorem}
\label{t2.3}Assumptions are those of Theorem \emph{\ref{t2.2}}. Assume
moreover that $p\geq 2$ and that there exists a function $u_{0}\in
L^{p}(0,T;W_{0}^{1,p}(Q))$ such that $u_{\varepsilon }\rightarrow u_{0}$ in $%
L^{2}(Q_{T})$ as $E\ni \varepsilon \rightarrow 0$. Then there exists a
subsequence $E^{\prime }$ of $E$ and a function $u_{1}\in L^{p}(Q_{T}\times
Z;W_{\#\rho }^{1,p}(Y))$ such that, as $E^{\prime }\ni \varepsilon
\rightarrow 0$,
\begin{equation}
\frac{\partial u_{\varepsilon }}{\partial x_{j}}\rightarrow \frac{\partial
u_{0}}{\partial x_{j}}+\frac{\partial u_{1}}{\partial y_{j}}\text{ in }%
L^{p}(Q_{T})\text{-2s }(1\leq j\leq N).  \label{2.2}
\end{equation}
\end{theorem}

\begin{proof}
Let $u_{1}^{\#}\in L^{p}(Q_{T}\times Z;W_{\text{per}}^{1,p}(Y))$ be such
that Theorem \ref{t2.2} holds with $u_{1}^{\#}$ in place of $u_{1}$ in that
theorem. Set
\begin{equation*}
u_{1}(x,t,y,\tau )=u_{1}^{\#}(x,t,y,\tau )-\frac{1}{\int_{Y}\rho (y)dy}%
\int_{Y}\rho (y)u_{1}^{\#}(x,t,y,\tau )dy
\end{equation*}%
for $(x,t,y,\tau )\in Q_{T}\times Y\times Z$. Then $u_{1}\in
L^{p}(Q_{T}\times Z;W_{\#\rho }^{1,p}(Y))$ and moreover $\partial
u_{1}/\partial y_{i}=\partial u_{1}^{\#}/\partial y_{i}$ ($1\leq i\leq N$),
so that (\ref{2.2}) holds.
\end{proof}

\begin{remark}
\label{r2.1}\emph{In case }$\rho \equiv 1$\emph{, we retrieve the result of
\cite{NgWou} since in that case }$W_{\#\rho
}^{1,p}(Y)=W_{\#}^{1,p}(Y):=\{u\in W_{\text{\emph{per}}}^{1,p}(Y):%
\int_{Y}udy=0\}$\emph{. Throughout the rest of the paper, we assume without
lost of generality that }$\int_{Y}\rho dy=1$\emph{.}
\end{remark}

\section{Statement of the problem: A priori estimates and compactness result
for the solution}

\subsection{Problem setting}

Let $Q$ be a bounded Lipschitz domain of $\mathbb{R}^{N}$ and $T$ a positive
real number. By $Q_{T}$ we denote the cylinder $Q\times (0,T)$. Our aim is
to study the asymptotic behavior of the sequence of solutions to (\ref{1.1}%
). We begin this section by setting the necessary conditions under which
such a study can be made possible. For instance, we assume that the
coefficients of (\ref{1.1}) are constrained as follows:

\begin{itemize}
\item[\textbf{A1}] The function $a:(x,t,y,\tau ,\lambda )\mapsto
a(x,t,y,\tau ,\lambda )$ from $\overline{Q}_{T}\times \mathbb{R}^{N}\times
\mathbb{R}\times \mathbb{R}^{N}$ into $\mathbb{R}^{N}$ satisfies the
properties that:
\begin{equation}
\text{For each fixed }(x,t)\in \overline{Q}_{T}\text{ and }\lambda \in
\mathbb{R}^{N}\text{, }a(x,t,\cdot ,\cdot ,\lambda )\text{ is measurable}
\label{3.1}
\end{equation}%
\begin{equation}
\begin{array}{l}
a(x,t,y,\tau ,0)=0\text{ almost everywhere (a.e.) in }(y,\tau )\in \mathbb{R}%
^{N}\times \mathbb{R} \\
\text{and for all }(x,t)\in \overline{Q}_{T}\text{.}%
\end{array}
\label{3.2}
\end{equation}%
\begin{equation}
\begin{array}{l}
\text{There are three constants }c_{0},\,c_{1},\,c_{2}>0\text{ and a
continuity modulus }\omega \\
\text{(i.e., a nondecreasing continuous function on }[0,+\infty )\text{ such
that } \\
\omega (0)=0,\omega (r)>0\text{\ if }r>0\text{, and }\omega (r)=1\text{ if }%
r>1\text{) such that a.e. in} \\
(y,\tau )\in \mathbb{R}^{N}\times \mathbb{R}\text{,} \\
\text{(i) }\left( a(x,t,y,\tau ,\lambda )-a(x,t,y,\tau ,\lambda ^{\prime
})\right) \cdot \left( \lambda -\lambda ^{\prime }\right) \geq c_{1}\left|
\lambda -\lambda ^{\prime }\right| ^{p} \\
\text{(ii) }\left| a(x,t,y,\tau ,\lambda )\right| \leq c_{2}(1+\left|
\lambda \right| ^{p-1}) \\
\text{(iii) }\left| a(x,t,y,\tau ,\lambda )-a(x^{\prime },t^{\prime },y,\tau
,\lambda ^{\prime })\right| \\
\;\;\;\;\leq \omega (\left| x-x^{\prime }\right| +\left| t-t^{\prime
}\right| )(1+\left| \lambda \right| ^{p-1}+\left| \lambda ^{\prime }\right|
^{p-1})+c_{0}\left( 1+\left| \lambda \right| +\left| \lambda ^{\prime
}\right| \right) ^{p-2}\left| \lambda -\lambda ^{\prime }\right| \\
\text{for all }(x,t),(x^{\prime },t^{\prime })\in \overline{Q}_{T}\text{ and
all }\lambda ,\lambda ^{\prime }\in \mathbb{R}^{N}\text{, where the dot} \\
\text{denotes the usual Euclidean inner product in }\mathbb{R}^{N}\text{ and
}\left| \cdot \right| \text{ the associated} \\
\text{norm.}%
\end{array}
\label{3.3}
\end{equation}

\item[\textbf{A2}] \textbf{Lipschitz continuity}. The function $g$ is
continuous on $\mathbb{R}^{N}\times \mathbb{R}\times \mathbb{R}$ and there
is $C>0$ such that for any $(y,\tau )\in \mathbb{R}^{N+1}$ and $r$, $r_{1}$,
$r_{2}\in \mathbb{R}$
\begin{equation*}
\begin{array}{l}
\left| \partial _{r}g(y,\tau ,r)\right| \leq C \\
\left| \partial _{r}g(y,\tau ,r_{1})-\partial _{r}g(y,\tau ,r_{2})\right|
\leq C\left| r_{1}-r_{2}\right| (1+\left| r_{1}\right| +\left| r_{2}\right|
)^{-1}\text{.}%
\end{array}%
\end{equation*}

\item[\textbf{A3}] \textbf{Equilibrium condition}. We assume that $0$ is a
possible equilibrium solution of (\ref{1.1}), that is, $g(y,\tau ,0)=0$ for
any $(y,\tau )\in \mathbb{R}^{N+1}$.

\item[\textbf{A4}] \textbf{Positivity}. The density function $\rho \in
L^{\infty }(\mathbb{R}^{N})$ and there exists $\Lambda >0$ such that
\begin{equation*}
\Lambda ^{-1}\leq \rho (y)\leq \Lambda \text{\ for a.e. }y\in \mathbb{R}^{N}.
\end{equation*}%
We also assume without loss of generality that
\begin{equation*}
\int_{Y}\rho (y)dy=1.
\end{equation*}

\item[\textbf{A5}] \textbf{Periodicity hypothesis}. The density function $%
\rho $ is $Y$-periodic, the function $(y,\tau )\mapsto a(x,t,y,\tau ,\lambda
)$ is $Y\times Z$-periodic for any fixed $x,t,\lambda $. We assume also that
$g(\cdot ,\cdot ,r)\in \mathcal{C}_{\text{per}}(Y\times Z)$ for any $r\in
\mathbb{R}$ with $\int_{Y}g(y,\tau ,r)dy=0$ for all $(\tau ,r)\in \mathbb{R}%
^{2}$. We easily infer from the Fredholm alternative the existence of a
unique $R(\cdot ,\cdot ,r)\in \mathcal{C}_{\text{per}}(Y\times Z)$ such that
$\Delta _{y}R(\cdot ,\cdot ,r)=g(\cdot ,\cdot ,r)$ and $\int_{Y}R(\cdot
,\tau ,r)dy=0$ for all $\tau $, $r\in \mathbb{R}$, where $\Delta _{y}$
stands for the Laplacian with respect to the variable $y$. Moreover $R(\cdot
,\cdot ,r)$ is at least twice differentiable with respect to $y$. Let $%
G(y,\tau ,r)=D_{y}R(y,\tau ,r)$. Thanks to \textbf{A2} and \textbf{A3} we
see that
\begin{equation}
\left| G(y,\tau ,r)\right| \leq C\left| r\right| \text{, }\left| \partial
_{r}G(y,\tau ,r)\right| \leq C\text{,}  \label{3.4}
\end{equation}%
\begin{equation}
\left| \partial _{r}G(y,\tau ,r_{1})-\partial _{r}G(y,\tau ,r_{2})\right|
\leq C\left| r_{1}-r_{2}\right| (1+\left| r_{1}\right| +\left| r_{2}\right|
)^{-1}  \label{3.5}
\end{equation}%
where $\partial _{r}G$ denotes the partial derivative of $G$ with respect to
$r$.
\end{itemize}

As regards the definition of the trace functions $(x,t)\mapsto
a(x,t,x/\varepsilon ,t/\varepsilon ^{k},Du_{\varepsilon }(x,t))$, $%
(x,t)\mapsto g(x/\varepsilon ,t/\varepsilon ^{k},u_{\varepsilon }(x,t))$ and
$x\mapsto \rho (x/\varepsilon )$ here denoted respectively by $%
a^{\varepsilon }(\cdot ,Du_{\varepsilon })$, $g^{\varepsilon
}(u_{\varepsilon })$ and $\rho ^{\varepsilon }$, this has been extensively
discussed in many papers (see e.g. \cite{EJDE, AMPA}). These functions are
well-defined and satisfy properties of the same type as in \textbf{A1}-%
\textbf{A4}. Due to both the positivity assumption on the density function $%
\rho $ and the Lipschitzity of the function $g(y,\tau ,\cdot )$, one can
show in a\ standard fashion that the problem (\ref{1.1}) admits a unique
solution $u_{\varepsilon }$, which moreover belongs to the space $%
L^{p}(0,T;W_{0}^{1,p}(Q))\cap \mathcal{C}(0,T;L^{2}(Q))$; see e.g., \cite%
{Alt, Paronetto}.

\subsection{A priori estimates and compactness}

We will denote by $\left( \cdot ,\cdot \right) $ the duality pairing between
$W_{0}^{1,p}(Q)$ and its topological dual $W^{-1,p^{\prime }}(Q)$. The
symbols $\left\vert \cdot \right\vert _{L^{p}}$ and $\left\Vert \cdot
\right\Vert $ will stand for the respective norms of $L^{p}(Q)$ and $%
W_{0}^{1,p}(Q)$. Throughout $C$ will denote a generic constant independent
of $\varepsilon $. The following uniform a priori estimates hold.

\begin{lemma}
\label{l3.1}Under assumptions \textbf{A1}-\textbf{A5} the following
estimates hold true for $2\leq p<\infty $:
\begin{equation}
\sup_{0\leq t\leq T}\left\vert u_{\varepsilon }(t)\right\vert
_{L^{2}}^{2}\leq C,  \label{3.6}
\end{equation}%
\begin{equation}
\int_{0}^{T}\left\Vert u_{\varepsilon }(t)\right\Vert ^{p}dt\leq C
\label{3.7}
\end{equation}%
where $C$ is a positive constant which does not depend on $\varepsilon $.
\end{lemma}

\begin{proof}
We have $u_{\varepsilon }\in L^{p}(0,T;W_{0}^{1,p}(Q))\cap \mathcal{C}%
(0,T;L^{2}(Q))$ and the following energy equation holds:
\begin{eqnarray}
&&\left| (\rho ^{\varepsilon })^{\frac{1}{2}}u_{\varepsilon }(t)\right|
_{L^{2}}^{2}-\left| (\rho ^{\varepsilon })^{\frac{1}{2}}u^{0}\right|
_{L^{2}}^{2}+2\int_{0}^{t}\int_{Q}a^{\varepsilon }(\cdot ,Du_{\varepsilon
}(s))\cdot Du_{\varepsilon }(s)dxds  \label{3.8} \\
&=&2\int_{0}^{t}\int_{Q}\frac{1}{\varepsilon }g^{\varepsilon
}(u_{\varepsilon }(s))u_{\varepsilon }(s)dxds.  \notag
\end{eqnarray}%
But using the representation $G(y,\tau ,r)=D_{y}R(y,\tau ,r)$ obtained
above, we get
\begin{equation*}
\frac{1}{\varepsilon }g\left( \frac{x}{\varepsilon },\frac{t}{\varepsilon
^{k}},u_{\varepsilon }\right) =\Div G\left( \frac{x}{\varepsilon },\frac{t}{%
\varepsilon ^{k}},u_{\varepsilon }\right) -\partial _{r}G\left( \frac{x}{%
\varepsilon },\frac{t}{\varepsilon ^{k}},u_{\varepsilon }\right) \cdot
Du_{\varepsilon }.
\end{equation*}%
Putting the above expression in (\ref{3.8}) we obtain
\begin{eqnarray*}
&&\left| (\rho ^{\varepsilon })^{\frac{1}{2}}u_{\varepsilon }(t)\right|
_{L^{2}}^{2}+2\int_{0}^{t}\int_{Q}a^{\varepsilon }(\cdot ,Du_{\varepsilon
})\cdot Du_{\varepsilon }dxds \\
&=&-2\int_{0}^{t}\int_{Q}G^{\varepsilon }(u_{\varepsilon })\cdot
Du_{\varepsilon }dxds-2\int_{0}^{t}\int_{Q}(\partial _{r}G^{\varepsilon
}(u_{\varepsilon })\cdot Du_{\varepsilon })u_{\varepsilon }dxds+\left| (\rho
^{\varepsilon })^{\frac{1}{2}}u^{0}\right| _{L^{2}}^{2},
\end{eqnarray*}%
where $G^{\varepsilon }(u_{\varepsilon })$ and $\partial _{r}G^{\varepsilon
}(u_{\varepsilon })$ are defined exactly as $g^{\varepsilon }(u_{\varepsilon
})$. Thanks to Assumptions \textbf{A1}-\textbf{A5} one can easily see that
\begin{eqnarray*}
\Lambda ^{-1}\left| u_{\varepsilon }(t)\right|
_{L^{2}}^{2}+2c_{1}\int_{0}^{t}\int_{Q}\left| Du_{\varepsilon }\right|
^{p}dxds &\leq &2C\int_{0}^{t}\int_{Q}\left| u_{\varepsilon }\right| \left|
Du_{\varepsilon }\right| dxds \\
&&+2C\int_{0}^{t}\int_{Q}\left| u_{\varepsilon }\right| \left|
Du_{\varepsilon }\right| dxds+\Lambda \left| u^{0}\right| _{L^{2}}^{2}.
\end{eqnarray*}%
But by Young's inequality, we have, for any positive $\delta $,
\begin{equation*}
4C\int_{0}^{t}\int_{Q}\left| u_{\varepsilon }\right| \left| Du_{\varepsilon
}\right| dxds\leq \int_{0}^{t}\int_{Q}\left( \frac{4C\delta ^{-p^{\prime }}}{%
p^{\prime }}\left| u_{\varepsilon }\right| ^{p^{\prime }}+\frac{4C\delta ^{p}%
}{p}\left| Du_{\varepsilon }\right| ^{p}\right) dxds.
\end{equation*}%
Choosing $\delta $ in such a way that $\frac{4C\delta ^{p}}{p}=c_{1}$ we get
\begin{equation*}
\Lambda ^{-1}\left| u_{\varepsilon }(t)\right|
_{L^{2}}^{2}+c_{1}\int_{0}^{t}\int_{Q}\left| Du_{\varepsilon }\right|
^{p}dxds\leq C\int_{0}^{t}\left| u_{\varepsilon }\right| _{L^{p^{\prime
}}}^{p^{\prime }}ds+K
\end{equation*}%
where $K=\Lambda \left| u^{0}\right| _{L^{2}}^{2}$. But as $p^{\prime }\leq
2 $, we have $\left| u_{\varepsilon }\right| _{L^{p^{\prime }}}^{p^{\prime
}}\leq C\left| u_{\varepsilon }\right| _{L^{2}}^{p^{\prime }}$, thus
\begin{equation}
\Lambda ^{-1}\left| u_{\varepsilon }(t)\right|
_{L^{2}}^{2}+c_{1}\int_{0}^{t}\int_{Q}\left| Du_{\varepsilon }\right|
^{p}dxds\leq C\int_{0}^{t}\left| u_{\varepsilon }\right| _{L^{2}}^{p^{\prime
}}ds,  \label{3.9}
\end{equation}%
hence
\begin{equation*}
\Lambda ^{-1}\left| u_{\varepsilon }(t)\right| _{L^{2}}^{2}\leq
C\int_{0}^{t}\left| u_{\varepsilon }\right| _{L^{2}}^{p^{\prime }}ds.
\end{equation*}%
But, since $p^{\prime }\leq 2$, there is a positive constant $k_{1}$
independent of $\varepsilon $ such that $\left| u_{\varepsilon }(t)\right|
_{L^{2}}^{p^{\prime }}\leq k_{1}(1+\left| u_{\varepsilon }(t)\right|
_{L^{2}}^{2})$. Thus
\begin{equation*}
\left| u_{\varepsilon }(t)\right| _{L^{2}}^{2}\leq C+C\int_{0}^{t}\left|
u_{\varepsilon }(s)\right| _{L^{2}}^{2}ds
\end{equation*}%
where here, $C=C(T)>0$. By the application of Gronwall inequality we get at
once (\ref{3.6}). We then deduce (\ref{3.7}) from (\ref{3.9}).
\end{proof}

The next result should be of capital interest for the sequel.

\begin{proposition}
\label{p3.1}The family $(u_{\varepsilon })_{\varepsilon >0}$ is relatively
compact in the space $L^{2}(Q_{T})$.
\end{proposition}

\begin{proof}
It follows from Eq. (\ref{1.1}) that
\begin{eqnarray*}
\left\Vert \rho ^{\varepsilon }\frac{\partial u_{\varepsilon }}{\partial t}%
\right\Vert _{L^{p^{\prime }}(0,T;W^{-1,p^{\prime }}(Q))}^{p^{\prime }}
&\leq &C\int_{0}^{T}\left\Vert \Div a^{\varepsilon }(\cdot ,Du_{\varepsilon
})\right\Vert _{W^{-1,p^{\prime }}(Q)}^{p^{\prime }}dt \\
&&+C\int_{0}^{T}\left\Vert \frac{1}{\varepsilon }g^{\varepsilon
}(u_{\varepsilon })\right\Vert _{W^{-1,p^{\prime }}(Q)}^{p^{\prime }}dt,
\end{eqnarray*}%
and, thanks to (\ref{3.7}) we easily get
\begin{equation}
\int_{0}^{T}\left\Vert \Div a^{\varepsilon }(\cdot ,Du_{\varepsilon
})\right\Vert _{W^{-1,p^{\prime }}(Q)}^{p^{\prime }}dt\leq C.  \label{3.11}
\end{equation}%
Next
\begin{equation*}
\left\Vert \frac{1}{\varepsilon }g^{\varepsilon }(u_{\varepsilon
})\right\Vert _{W^{-1,p^{\prime }}(Q)}=\sup_{\substack{ \phi \in
W_{0}^{1,p}(Q)  \\ \left\Vert \phi \right\Vert =1}}\left\vert
\int_{Q}G^{\varepsilon }(u_{\varepsilon })\cdot D\phi dx+\int_{Q}(\partial
_{r}G^{\varepsilon }(u_{\varepsilon })\cdot Du_{\varepsilon })\phi
dx\right\vert .
\end{equation*}%
By using condition \textbf{A5} (see especially (\ref{3.4})) associated to
the Poincar\'{e}'s inequality we get that
\begin{eqnarray*}
\left\Vert \frac{1}{\varepsilon }g^{\varepsilon }(u_{\varepsilon
})\right\Vert _{W^{-1,p^{\prime }}(Q)} &\leq &\sup_{\phi \in
W_{0}^{1,p}(Q),\left\Vert \phi \right\Vert =1}(C\left\vert u_{\varepsilon
}(s)\right\vert _{L^{2}}+C\left\Vert u_{\varepsilon }(s)\right\Vert
\left\vert \phi \right\vert ) \\
&\leq &C\left\vert u_{\varepsilon }(s)\right\vert _{L^{2}}+C\left\Vert
u_{\varepsilon }(s)\right\Vert .
\end{eqnarray*}%
It therefore follows from the estimates (\ref{3.6})-(\ref{3.7}) and the H%
\"{o}lder's inequality that
\begin{equation}
\int_{0}^{T}\left\Vert \frac{1}{\varepsilon }g^{\varepsilon }(u_{\varepsilon
})\right\Vert _{W^{-1,p^{\prime }}(Q)}^{p^{\prime }}dt\leq C.  \label{3.12}
\end{equation}%
Thus we infer from (\ref{3.11})-(\ref{3.12}) that
\begin{equation*}
\left\Vert \rho ^{\varepsilon }\frac{\partial u_{\varepsilon }}{\partial t}%
\right\Vert _{L^{p^{\prime }}(0,T;W^{-1,p^{\prime }}(Q))}\leq C.
\end{equation*}%
Hence, setting
\begin{equation*}
\mathcal{W}_{\varepsilon }=\{u\in L^{p}(0,T;W_{0}^{1,p}(Q)):(\rho
^{\varepsilon }u)^{\prime }\in L^{p^{\prime }}(0,T;W^{-1,p^{\prime }}(Q))\}
\end{equation*}%
which is a Banach space under the norm
\begin{equation*}
\left\Vert u\right\Vert _{\mathcal{W}_{\varepsilon }}=\left\Vert
u\right\Vert _{L^{p}(0,T;W_{0}^{1,p}(Q))}+\left\Vert (\rho ^{\varepsilon
}u)^{\prime }\right\Vert _{L^{p^{\prime }}(0,T;W^{-1,p^{\prime }}(Q))},
\end{equation*}%
we have that
\begin{equation*}
\left\Vert u_{\varepsilon }\right\Vert _{\mathcal{W}_{\varepsilon }}\leq C%
\text{ for any }\varepsilon >0
\end{equation*}%
where $C$ is independent of $\varepsilon $. Since $\int_{Y}\rho dy\neq 0$,
we therefore deduce from \cite[Theorem 2.3]{Amar} that $(u_{\varepsilon
})_{\varepsilon >0}$ is relatively compact in $L^{2}(Q_{T})$.
\end{proof}

\section{Preliminary results}

Let $2\leq p<\infty $. The following Gelfand triplet
\begin{equation*}
W_{\#\rho }^{1,p}(Y)\subset L_{\#\rho }^{2}(Y)\subset (W_{\#\rho
}^{1,p}(Y))^{\prime }
\end{equation*}%
holds, with continuous embeddings, $(W_{\#\rho }^{1,p}(Y))^{\prime }$ being
the topological dual of $W_{\#\rho }^{1,p}(Y)$; this can be seen by showing
that the space $W_{\#\rho }^{1,p}(Y)$ is densely embedded in $L_{\#\rho
}^{2}(Y)$ (this follows by repeating the proof of the forthcoming Lemma \ref%
{l4.2}). It is also a fact that the topological dual of $L_{\text{per}%
}^{p}(Z;W_{\#\rho }^{1,p}(Y))$ is $L_{\text{per}}^{p^{\prime }}(Z;[W_{\#\rho
}^{1,p}(Y)]^{\prime })$; this can be easily seen from the fact that $%
W_{\#\rho }^{1,p}(Y)$ is reflexive and $L_{\text{per}}^{p}(Z;W_{\#\rho
}^{1,p}(Y))$ is isometrically isomorphic to $L^{p}(\mathbb{T};W_{\#\rho
}^{1,p}(Y))$ where $\mathbb{T}$ is the 1-dimensional torus. We denote by $%
\left( ,\right) $ (resp. $[,]$) the duality pairing between $W_{\#\rho
}^{1,p}(Y)$ (resp. $L_{\text{per}}^{p}(Z;W_{\#\rho }^{1,p}(Y))$) and $%
[W_{\#\rho }^{1,p}(Y)]^{\prime }$ (resp. $L_{\text{per}}^{p^{\prime
}}(Z;[W_{\#\rho }^{1,p}(Y)]^{\prime })$). For the above reasons we have,
\begin{equation*}
\left[ u,v\right] =\int_{0}^{1}\left( u(\tau ),v(\tau )\right) d\tau
\end{equation*}%
for $u\in L_{\text{per}}^{p^{\prime }}(Z;[W_{\#\rho }^{1,p}(Y)]^{\prime })$
and $v\in L_{\text{per}}^{p}(Z;W_{\#\rho }^{1,p}(Y))$, and
\begin{equation*}
\left( u,\varphi \right) =\int_{Y}u(y)\varphi (y)dy
\end{equation*}%
for all $u\in L_{\#\rho }^{2}(Y)$ and $\varphi \in W_{\#\rho }^{1,p}(Y)$.

The following important density result will be used throughout the paper.

\begin{lemma}
\label{l4.2}The space
\begin{equation*}
\mathcal{D}_{\#\rho }(Y)=\left\{ u\in \mathcal{D}_{\text{\emph{per}}%
}(Y):\int_{Y}\rho udy=0\right\}
\end{equation*}%
is dense in $W_{\#\rho }^{1,p}(Y)$.
\end{lemma}

\begin{proof}
Let $L$ be a continuous linear functional on $W_{\#\rho }^{1,p}(Y)$
verifying $L(v)=0$ for all $v\in \mathcal{D}_{\#\rho }(Y)$. We need to check
that $L(v)=0$ for all $v\in W_{\#\rho }^{1,p}(Y)$. By the Hahn-Banach
theorem, $L$ extends to a (possibly non unique!) continuous linear
functional $\widetilde{L}$ on $W_{\text{per}}^{1,p}(Y)$ and so, there exists
$(u_{i})_{0\leq i\leq N}\subset (L_{\text{per}}^{p^{\prime }}(Y))^{N+1}$
such that
\begin{equation*}
\widetilde{L}(v)=\int_{Y}u_{0}vdy+\sum_{i=1}^{N}\int_{Y}u_{i}\frac{\partial v%
}{\partial y_{i}}dy\text{\ for all }v\in W_{\text{per}}^{1,p}(Y).
\end{equation*}%
Let $v\in \mathcal{D}_{\text{per}}(Y)$; since $v-M_{y}(\rho v)\in \mathcal{D}%
_{\#\rho }(Y)$ (where $M_{y}(\rho v)=\int_{Y}\rho vdy$), we have $\widetilde{%
L}(v-M_{y}(\rho v))=L(v-M_{y}(\rho v))=0$, or equivalently,
\begin{equation}
L(v)=M_{y}(\rho v)\int_{Y}u_{0}dy\text{\ for all }v\in \mathcal{D}_{\text{per%
}}(Y).  \label{4.1}
\end{equation}%
But the linear functional $v\mapsto M_{y}(\rho v)$ defined on $\mathcal{D}_{%
\text{per}}(Y)$, is continuous on $\mathcal{D}_{\text{per}}(Y)$ endowed with
the $W_{\text{per}}^{1,p}(Y)$-norm, so that (\ref{4.1}) still holds true for
$v\in W_{\text{per}}^{1,p}(Y)$ (by the density of $\mathcal{D}_{\text{per}%
}(Y)$ in $W_{\text{per}}^{1,p}(Y)$). Therefore taking $v\in W_{\#\rho
}^{1,p}(Y)$ we get $L(v)=0$. This completes the proof.
\end{proof}

The following obvious result will be used in the sequel.

\begin{lemma}
\label{l4.1}Let $u\in \mathcal{D}_{\text{\emph{per}}}^{\prime }(Y\times Z)$.
We still write $u$ for $\left. u\right\vert _{\mathcal{D}_{\text{\emph{per}}%
}(Z)\otimes \lbrack \mathcal{D}_{\#\rho }(Y)]}$ (the restriction of $u$ to $%
\mathcal{D}_{\text{\emph{per}}}(Z)\otimes \lbrack \mathcal{D}_{\#\rho }(Y)]$%
). Assume $u$ is continuous on $\mathcal{D}_{\text{\emph{per}}}(Z)\otimes
\lbrack \mathcal{D}_{\#\rho }(Y)]$ under the $L_{\text{\emph{per}}%
}^{p}(Z;W_{\#\rho }^{1,p}(Y))$-norm. Then $u\in L_{\text{\emph{per}}%
}^{p^{\prime }}(Z;[W_{\#\rho }^{1,p}(Y)]^{\prime })$ and further
\begin{equation*}
\left\langle u,\varphi \right\rangle =\int_{0}^{1}\left( u(\tau ),\varphi
(\cdot ,\tau )\right) d\tau \;\;\;\;\;\;\;\;\;\;\;\;\;\;\;
\end{equation*}%
for all $\varphi \in \mathcal{D}_{\text{\emph{per}}}(Z)\otimes \lbrack
\mathcal{D}_{\#\rho }(Y)]$, where $\left\langle ,\right\rangle $ denotes the
duality pairing between $\mathcal{D}_{\text{\emph{per}}}^{\prime }(Y\times
Z) $ and $\mathcal{D}_{\text{\emph{per}}}(Y\times Z)$, whereas the
right-hand side denotes the product of $u$ and $\varphi $ in the duality
between $L_{\text{\emph{per}}}^{p^{\prime }}(Z;[W_{\#\rho
}^{1,p}(Y)]^{\prime })$ and $L_{\text{\emph{per}}}^{p}(Z;W_{\#\rho
}^{1,p}(Y))$ as stated above.
\end{lemma}

The next result will prove very efficient in the homogenization process in
the case when $k=2$.

\begin{lemma}
\label{l4.3}Let $\psi \in \mathcal{C}_{0}^{\infty }(Q_{T})\otimes (\mathcal{D%
}_{\text{\emph{per}}}(Z)\otimes \lbrack \mathcal{D}_{\#\rho }(Y)])$. Let $%
(u_{\varepsilon })_{\varepsilon \in E}$, $E^{\prime }$ and $(u_{0},u_{1})$
be as in Theorem \emph{\ref{t2.3}}. Then, as $E^{\prime }\ni \varepsilon
\rightarrow 0$
\begin{equation*}
\int_{Q_{T}}\frac{1}{\varepsilon }u_{\varepsilon }\rho ^{\varepsilon }\psi
^{\varepsilon }dxdt\rightarrow \int_{Q_{T}}\left[ \rho u_{1}(x,t),\psi (x,t)%
\right] dxdt
\end{equation*}%
where $\psi ^{\varepsilon }(x,t)=\psi (x,t,x/\varepsilon ,t/\varepsilon
^{k}) $ for $(x,t)\in Q_{T}$.
\end{lemma}

\begin{proof}
We recall that the space $\mathcal{D}_{\#\rho }(Y)$ consists of those
functions $\psi $ in $\mathcal{D}_{\text{per}}(Y)$ with the property $%
\int_{Y}\rho \psi dy=0$. With this in mind, let $\psi $ be as in the
statement of the lemma. Since $\rho \psi $ is in $\mathcal{C}_{0}^{\infty
}(Q_{T})\otimes (\mathcal{D}_{\text{per}}(Z)\otimes \lbrack \mathcal{D}_{%
\text{per}}(Y)])$ and verifies $\int_{Y}\rho \psi dy=0$, the result follows
at once by the application of \cite[Lemma 3.4]{NgWou}.
\end{proof}

Let $\mathcal{R}:L_{\text{per}}^{2}(Y)\rightarrow L_{\text{per}}^{2}(Y)$ be
defined by $\mathcal{R}u=\rho u$. Then $\mathcal{R}$ is a non-negative and
bounded linear self-adjoint operator. By the positivity of $\rho $, its
kernel is reduced to $0$. We denote by $L_{\rho }^{2}(Y)$ the completion of $%
L_{\text{per}}^{2}(Y)$ with respect to the norm $\left\Vert u\right\Vert
_{+}=\left\Vert \rho ^{1/2}u\right\Vert _{L^{2}(Y)}$.

Now, for $u\in L_{\text{per}}^{2}(Z;L_{\text{per}}^{2}(Y))$ we define $%
\mathcal{R}u$ as follows:
\begin{equation*}
\mathcal{R}u(\tau )=\mathcal{R}(u(\tau ))\text{\ for a.e. }\tau \in Z=(0,1),
\end{equation*}%
and we get an operator $\mathcal{R}:L_{\text{per}}^{2}(Z;L_{\text{per}%
}^{2}(Y))\rightarrow L_{\text{per}}^{2}(Z;L_{\text{per}}^{2}(Y))$. Finally
let $\mathcal{V}=L_{\text{per}}^{p}(Z;W_{\#\rho }^{1,p}(Y))$ and its
topological dual $\mathcal{V}^{\prime }=L_{\text{per}}^{p^{\prime
}}(Z;[W_{\#\rho }^{1,p}(Y)]^{\prime })$. Viewing $(\mathcal{R})^{\prime
}\equiv \rho \partial /\partial \tau $ as an unbounded operator defined from
$\mathcal{V}$ into $\mathcal{V}^{\prime }$, its domain is
\begin{equation*}
\mathcal{W}=\left\{ v\in \mathcal{V}:\rho \frac{\partial v}{\partial \tau }%
\in \mathcal{V}^{\prime }\right\} .
\end{equation*}%
With the norm $\left\| v\right\| _{\mathcal{W}}=\left\| v\right\| _{\mathcal{%
V}}+\left\| \rho \frac{\partial v}{\partial \tau }\right\| _{\mathcal{V}%
^{\prime }}$, $\mathcal{W}$ is a Banach space and the following result holds.

\begin{proposition}
\label{p4.1}The space $\mathcal{W}$ is continuously embedded into $\mathcal{C%
}([0,1];L_{\rho }^{2}(Y))$, that is, there is a positive constant $c$ such
that
\begin{equation*}
\sup_{0\leq \tau \leq 1}\left\| \rho ^{\frac{1}{2}}u(\tau )\right\|
_{L^{2}(Y)}\leq c\left\| u\right\| _{\mathcal{W}}
\end{equation*}%
for all $u\in \mathcal{W}$. Moreover
\begin{equation}
\left[ \rho \frac{\partial u}{\partial \tau },v\right] =-\left[ \rho \frac{%
\partial v}{\partial \tau },u\right]  \label{4.2}
\end{equation}%
for all $u,v\in \mathcal{W}$.
\end{proposition}

\begin{proof}
The fact that $\mathcal{W}$ embeds continuously into $\mathcal{C}%
([0,1];L_{\rho }^{2}(Y))$ follows from \cite{Pankov}; see also \cite[%
Proposition 4.1]{Paronetto}. Still from the same references, we have that,
for $u,v\in \mathcal{W}$,
\begin{equation*}
\left[ \rho \frac{\partial u}{\partial \tau },v\right] +\left[ \rho \frac{%
\partial v}{\partial \tau },u\right] =\int_{Y}\rho u(1)v(1)dy-\int_{Y}\rho
u(0)v(0)dy,
\end{equation*}%
and by the $Z$-periodicity of $u$ and $v$, it follows that the right-hand
side of the above equality is zero, hence (\ref{4.2}).
\end{proof}

By repeating the proof of Lemma \ref{l4.2} one can show that the space $%
\mathcal{E}=\mathcal{D}_{\text{per}}(Z)\otimes \lbrack \mathcal{D}_{\#\rho
}(Y)]$ is dense in $\mathcal{W}$. The operator $\rho \partial /\partial \tau
$ will be useful in the homogenization process for the case $k=2$. This
being so, the Lemma \ref{l4.3} has a crucial corollary.

\begin{corollary}
\label{c4.1}Let the hypotheses be those of Lemma \emph{\ref{l4.3}}. Assume
moreover that $u_{1}\in \mathcal{W}$ and that $k=2$. Then, as $E^{\prime
}\ni \varepsilon \rightarrow 0$,%
\begin{equation*}
\int_{Q_{T}}\varepsilon u_{\varepsilon }\rho ^{\varepsilon }\frac{\partial
\psi ^{\varepsilon }}{\partial t}dxdt\rightarrow -\int_{Q_{T}}\left[ \rho
\frac{\partial u_{1}}{\partial \tau }(x,t),\psi (x,t)\right] dxdt.
\end{equation*}
\end{corollary}

\begin{proof}
We have
\begin{equation*}
\int_{Q_{T}}\varepsilon u_{\varepsilon }\rho ^{\varepsilon }\frac{\partial
\psi ^{\varepsilon }}{\partial t}dxdt=\varepsilon \int_{Q_{T}}u_{\varepsilon
}\rho ^{\varepsilon }\left( \frac{\partial \psi }{\partial t}\right)
^{\varepsilon }dxdt+\frac{1}{\varepsilon }\int_{Q_{T}}u_{\varepsilon }\rho
^{\varepsilon }\left( \frac{\partial \psi }{\partial \tau }\right)
^{\varepsilon }dxdt.
\end{equation*}%
Since $\frac{\partial \psi }{\partial \tau }$ is in $\mathcal{C}_{0}^{\infty
}(Q_{T})\otimes (\mathcal{D}_{\text{per}}(Z)\otimes \lbrack \mathcal{D}%
_{\#\rho }(Y)])$, we infer from Lemma \ref{l4.3} that, as $E^{\prime }\ni
\varepsilon \rightarrow 0$,
\begin{equation*}
\int_{Q_{T}}\varepsilon u_{\varepsilon }\rho ^{\varepsilon }\frac{\partial
\psi ^{\varepsilon }}{\partial t}dxdt\rightarrow \int_{Q_{T}}\left[
\int_{Z}\left( \rho u_{1}(x,t,\cdot ,\tau ),\frac{\partial \psi }{\partial
\tau }(x,t,\cdot ,\tau ))\right) d\tau \right] dxdt.
\end{equation*}%
But
\begin{eqnarray*}
\int_{Z}\left( \rho u_{1}(x,t,\cdot ,\tau ),\frac{\partial \psi }{\partial
\tau }(x,t,\cdot ,\tau ))\right) d\tau &=&\left[ \rho u_{1}(x,t,\cdot ,\cdot
),\frac{\partial \psi }{\partial \tau }(x,t,\cdot ,\cdot ))\right] \\
&=&-\left[ \rho \frac{\partial u_{1}}{\partial \tau }(x,t,\cdot ,\cdot
),\psi (x,t,\cdot ,\cdot ))\right] ,
\end{eqnarray*}%
where in the last equality, we have used (\ref{4.2}) (see Proposition \ref%
{p4.1}).
\end{proof}

We will also need the following

\begin{lemma}
\label{l4.4}Let $g:\mathbb{R}_{y}^{N}\times \mathbb{R}_{\tau }\times \mathbb{%
R}_{r}\rightarrow \mathbb{R}$ be a function verifying the following
conditions:

\begin{itemize}
\item[(i)] $\left| \partial _{r}g(y,\tau ,u)\right| \leq C$

\item[(ii)] $g(\cdot ,\cdot ,r)\in \mathcal{C}_{\text{\emph{per}}}(Y\times
Z) $.
\end{itemize}

\noindent Let $(u_{\varepsilon })_{\varepsilon }$ be a sequence in $%
L^{2}(Q_{T})$ such that $u_{\varepsilon }\rightarrow u_{0}$ in $L^{2}(Q_{T})$
as $\varepsilon \rightarrow 0$, where $u_{0}\in L^{2}(Q_{T})$. Then, setting
$g^{\varepsilon }(u_{\varepsilon })(x,t)=g(x/\varepsilon ,t/\varepsilon
^{k},u_{\varepsilon }(x,t))$ we have, as $\varepsilon \rightarrow 0$,
\begin{equation*}
g^{\varepsilon }(u_{\varepsilon })\rightarrow g(\cdot ,\cdot ,u_{0})\text{
in }L^{2}(Q_{T})\text{-2s.}
\end{equation*}
\end{lemma}

\begin{proof}
Assumption (i) implies the Lipschitz condition
\begin{equation}
\left| g(y,\tau ,r_{1})-g(y,\tau ,r_{2})\right| \leq C\left|
r_{1}-r_{2}\right| \text{\ for all }y,\tau ,r_{1},r_{2}.  \label{4.3}
\end{equation}%
Next, observe that from (ii) and (\ref{4.3}), the function $(x,t,y,\tau
)\mapsto g(y,\tau ,u_{0}(x,t))$ lies in $L^{2}(Q_{T};\mathcal{C}_{\text{per}%
}(Y\times Z))$, so that we have $g^{\varepsilon }(u_{0})\rightarrow g(\cdot
,\cdot ,u_{0})$ in $L^{2}(Q_{T})$-2s as $\varepsilon \rightarrow 0$. Now,
for $f\in L^{2}(Q_{T};\mathcal{C}_{\text{per}}(Y\times Z))$,
\begin{eqnarray*}
&&\int_{Q_{T}}g^{\varepsilon }(u_{\varepsilon })f^{\varepsilon
}dxdt-\iint_{Q_{T}\times Y\times Z}g(\cdot ,\cdot ,u_{0})fdxdtdyd\tau \\
&=&\int_{Q_{T}}(g^{\varepsilon }(u_{\varepsilon })-g^{\varepsilon
}(u_{0}))f^{\varepsilon }dxdt+\int_{Q_{T}}g^{\varepsilon
}(u_{0})f^{\varepsilon }dxdt \\
&&-\iint_{Q_{T}\times Y\times Z}g(\cdot ,\cdot ,u_{0})fdxdtdyd\tau .
\end{eqnarray*}%
Using the inequality
\begin{equation*}
\left| \int_{Q_{T}}(g^{\varepsilon }(u_{\varepsilon })-g^{\varepsilon
}(u_{0}))f^{\varepsilon }dxdt\right| \leq C\left\| u_{\varepsilon
}-u_{0}\right\| _{L^{2}(Q_{T})}\left\| f^{\varepsilon }\right\|
_{L^{2}(Q_{T})}
\end{equation*}%
in conjunction with the above convergence results we get readily the result.
\end{proof}

\begin{remark}
\label{r4.1}\emph{From the Lipschitz property of the function }$g$\emph{\
above we may get more information on the limit of the sequence }$%
g^{\varepsilon }(u_{\varepsilon })$\emph{. Indeed, since }$\left\vert
g^{\varepsilon }(u_{\varepsilon })-g^{\varepsilon }(u_{0})\right\vert \leq
C\left\vert u_{\varepsilon }-u_{0}\right\vert $\emph{, we deduce the
following convergence result: }%
\begin{equation*}
g^{\varepsilon }(u_{\varepsilon })\rightarrow \widetilde{g}(u_{0})\text{\ in
}L^{2}(Q_{T})\text{ as }\varepsilon \rightarrow 0
\end{equation*}%
\emph{where }$\widetilde{g}(u_{0})(x,t)=\int_{Y\times Z}g(y,\tau
,u_{0}(x,t))dyd\tau $\emph{.}
\end{remark}

\bigskip We will need the following spaces:
\begin{equation*}
\mathbb{F}_{0}^{1,p}=L^{p}(0,T;W_{0}^{1,p}(Q))\times L^{p}(Q_{T};\mathcal{X})
\end{equation*}%
(where $\mathcal{X}$ is either $\mathcal{V}$\ or $\mathcal{W}$) and
\begin{equation*}
\mathcal{F}_{0}^{\infty }=\mathcal{C}_{0}^{\infty }(Q_{T})\times \lbrack
\mathcal{C}_{0}^{\infty }(Q_{T})\otimes \mathcal{E}]
\end{equation*}%
where we recall that $\mathcal{W}=\{v\in \mathcal{V}:\rho \partial
v/\partial \tau \in \mathcal{V}^{\prime }\}$ with $\mathcal{V}=L_{\text{per}%
}^{p}(Z;W_{\#\rho }^{1,p}(Y))$, and $\mathcal{E}=\mathcal{D}_{\text{per}%
}(Z)\otimes \lbrack \mathcal{D}_{\#\rho }(Y)]$. $\mathbb{F}_{0}^{1,p}$ is a
Banach space under the norm
\begin{equation*}
\left\| (u_{0},u_{1})\right\| _{\mathbb{F}_{0}^{1,p}}=\left\| u_{0}\right\|
_{L^{p}(\left( 0,T\right) ;W_{0}^{1,p}(Q))}+\left\| u_{1}\right\|
_{L^{p}(Q_{T};\mathcal{X})}
\end{equation*}%
with the further property that $\mathcal{F}_{0}^{\infty }$ is dense in $%
\mathbb{F}_{0}^{1,p}$; this obviously follows from Lemma \ref{l4.2}.

\section{Homogenization results}

\subsection{Global homogenized problem}

For a function $u\in L^{p}(0,T;W_{0}^{1,p}(Q))$, we shall denote by $%
u^{\prime }$ the partial derivative $\partial u/\partial t$ defined in a
distributional sense on $\mathcal{D}^{\prime }(Q_{T})$. Let $E$ be an
ordinary sequence of positive real numbers $\varepsilon $ converging to $0$
with $\varepsilon $. We assume throughout this section that $p\geq 2$. By
the strong relative compactness of the family $(u_{\varepsilon
})_{\varepsilon >0}$ (see Proposition \ref{p3.1}), there exist a subsequence
$E^{\prime }$ from $E$ and a function $u_{0}\in L^{p}(0,T;W_{0}^{1,p}(Q))$
such that, as $E^{\prime }\ni \varepsilon \rightarrow 0$,
\begin{equation}
u_{\varepsilon }\rightarrow u_{0}\text{ in }L^{2}(Q_{T}).  \label{5.1}
\end{equation}%
Let $u_{1}\in L^{p}(Q_{T}\times Z;W_{\#\rho }^{1,p}(Y))$ be the function
determined by the Theorem \ref{t2.3} such that, as $E^{\prime }\ni
\varepsilon \rightarrow 0$,
\begin{equation}
\frac{\partial u_{\varepsilon }}{\partial x_{j}}\rightarrow \frac{\partial
u_{0}}{\partial x_{j}}+\frac{\partial u_{1}}{\partial y_{j}}\text{ in }%
L^{p}(Q_{T})\text{-2s (}1\leq j\leq N\text{).}  \label{5.2}
\end{equation}%
The first important result of this section is the following

\begin{theorem}
\label{t5.1}The couple $(u_{0},u_{1})$ determined above by \emph{(\ref{5.1}%
)-(\ref{5.2})} solves the following variational problem:
\begin{equation}
\left\{
\begin{array}{l}
(u_{0},u_{1})\in L^{p}(0,T;W_{0}^{1,p}(Q))\times L^{p}(Q_{T};\mathcal{V}):
\\
\int_{0}^{T}\left( u_{0}^{\prime },v_{0}\right) dxdt+\iint_{Q_{T}\times
Y\times Z}a(\cdot ,Du_{0}+D_{y}u_{1})\cdot (Dv_{0}+D_{y}v_{1})dxdtdyd\tau \\
=\iint_{Q_{T}\times Y\times Z}g(y,\tau ,u_{0})v_{1}dxdtdyd\tau
-\iint_{Q_{T}\times Y\times Z}G(y,\tau ,u_{0})\cdot Dv_{0}dxdtdyd\tau \\
-\iint_{Q_{T}\times Y\times Z}\left( \partial _{r}G(y,\tau ,u_{0})\cdot
(Du_{0}+D_{y}u_{1})\right) v_{0}dxdtdyd\tau \\
\text{for all }(v_{0},v_{1})\in \mathcal{F}_{0}^{\infty }\text{, if }0<k<2;%
\end{array}%
\right.  \label{5.3}
\end{equation}%
\begin{equation}
\left\{
\begin{array}{l}
(u_{0},u_{1})\in L^{p}(0,T;W_{0}^{1,p}(Q))\times L^{p}(Q_{T};\mathcal{W}):
\\
\int_{0}^{T}\left( u_{0}^{\prime },v_{0}\right) dxdt+\int_{Q_{T}}\left[ \rho
\frac{\partial u_{1}}{\partial \tau },v_{1}\right] dxdt \\
=-\iint_{Q_{T}\times Y\times Z}a(\cdot ,Du_{0}+D_{y}u_{1})\cdot
(Dv_{0}+D_{y}v_{1})dxdtdyd\tau \\
+\iint_{Q_{T}\times Y\times Z}g(y,\tau ,u_{0})v_{1}dxdtdyd\tau
-\iint_{Q_{T}\times Y\times Z}G(y,\tau ,u_{0})\cdot Dv_{0}dxdtdyd\tau \\
-\iint_{Q_{T}\times Y\times Z}\left( \partial _{r}G(y,\tau ,u_{0})\cdot
(Du_{0}+D_{y}u_{1})\right) v_{0}dxdtdyd\tau \\
\text{for all }(v_{0},v_{1})\in \mathcal{F}_{0}^{\infty }\text{, if }k=2;%
\end{array}%
\right.  \label{5.4}
\end{equation}%
and
\begin{equation}
\left\{
\begin{array}{l}
(u_{0},u_{1})\in L^{p}(0,T;W_{0}^{1,p}(Q))\times L^{p}(Q_{T};W_{\#\rho
}^{1,p}(Y)): \\
\int_{0}^{T}\left( u_{0}^{\prime },v_{0}\right) dxdt+\iint_{Q_{T}\times Y}%
\overline{a}(\cdot ,Du_{0}+D_{y}u_{1})\cdot (Dv_{0}+D_{y}v_{1})dxdtdy \\
=\iint_{Q_{T}\times Y\times }\overline{g}(y,u_{0})v_{1}dxdtdy-\iint_{Q_{T}%
\times Y}\overline{G}(y,u_{0})\cdot Dv_{0}dxdtdy \\
-\iint_{Q_{T}\times Y}\left( \overline{\partial _{r}G}(y,u_{0})\cdot
(Du_{0}+D_{y}u_{1})\right) v_{0}dxdtdy \\
\text{for all }(v_{0},v_{1})\in \mathcal{C}_{0}^{\infty }(Q_{T})\times (%
\mathcal{C}_{0}^{\infty }(Q_{T})\otimes \mathcal{D}_{\#\rho }(Y))\text{, if }%
k>2\text{,}%
\end{array}%
\right.  \label{5.5}
\end{equation}%
where $\overline{a}(\cdot ,Du_{0}+D_{y}u_{1})=\int_{0}^{1}a(\cdot
,Du_{0}+D_{y}u_{1})d\tau $, $\overline{g}(y,u_{0})=\int_{0}^{1}g(y,\tau
,u_{0})d\tau $ (and a similar definition for $\overline{G}(y,u_{0})$ and $%
\overline{\partial _{r}G}(y,u_{0})$).
\end{theorem}

\begin{proof}
The proof will be done in three steps, according to the values of the
parameter $k.\bigskip $

\textit{Step 1:} \textbf{Case where} $0<k<2$. Let $\Phi =(\psi _{0},\psi
_{1})\in \mathcal{F}_{0}^{\infty }$, and define
\begin{equation*}
\Phi _{\varepsilon }(x,t)=\psi _{0}(x,t)+\varepsilon \psi _{1}\left( x,t,%
\frac{x}{\varepsilon },\frac{t}{\varepsilon ^{k}}\right) \text{, }(x,t)\in
Q_{T}\text{.}
\end{equation*}%
We recall that $\psi _{0}\in \mathcal{C}_{0}^{\infty }(Q_{T})$ and $\psi
_{1}\in \mathcal{C}_{0}^{\infty }(Q_{T})\otimes (\mathcal{D}_{\text{per}%
}(Z)\otimes \lbrack \mathcal{D}_{\#\rho }(Y)])$. Then $\Phi _{\varepsilon
}\in \mathcal{C}_{0}^{\infty }(Q_{T})$ and, using it as a test function in
the variational formulation of (\ref{1.1}), we get
\begin{equation}
-\int_{Q_{T}}\rho ^{\varepsilon }u_{\varepsilon }\frac{\partial \Phi
_{\varepsilon }}{\partial t}dxdt+\int_{Q_{T}}a^{\varepsilon }(\cdot
,Du_{\varepsilon })\cdot D\Phi _{\varepsilon }dxdt-\frac{1}{\varepsilon }%
\int_{Q_{T}}g^{\varepsilon }(u_{\varepsilon })\Phi _{\varepsilon }dxdt=0.
\label{5.6}
\end{equation}%
We consider the terms in (\ref{5.6}) respectively.

As regards the first term on the left-hand side of (\ref{5.6}), we have
\begin{eqnarray*}
\int_{Q_{T}}\rho ^{\varepsilon }u_{\varepsilon }\frac{\partial \Phi
_{\varepsilon }}{\partial t}dxdt &=&\int_{Q_{T}}\rho ^{\varepsilon
}u_{\varepsilon }\frac{\partial \psi _{0}}{\partial t}dxdt+\varepsilon
\int_{Q_{T}}\rho ^{\varepsilon }u_{\varepsilon }\left( \frac{\partial \psi
_{1}}{\partial t}\right) ^{\varepsilon }dxdt \\
&&+\varepsilon ^{2-k}\int_{Q_{T}}\frac{1}{\varepsilon }\rho ^{\varepsilon
}u_{\varepsilon }\left( \frac{\partial \psi _{1}}{\partial \tau }\right)
^{\varepsilon }dxdt.
\end{eqnarray*}%
In view of Lemma \ref{l4.3}, the integral $\int_{Q_{T}}\frac{1}{\varepsilon }%
\rho ^{\varepsilon }u_{\varepsilon }\left( \frac{\partial \psi _{1}}{%
\partial \tau }\right) ^{\varepsilon }dxdt$ converges (recall that $%
M_{y}(\rho \frac{\partial \psi _{1}}{\partial \tau })=0$). On the other
hand, since $\rho ^{\varepsilon }\rightarrow \int_{Y}\rho dy=1$ in $L^{2}(Q)$%
-weak and $u_{\varepsilon }\rightarrow u_{0}$ in $L^{2}(Q_{T})$, it is
immediate that
\begin{equation*}
\int_{Q_{T}}\rho ^{\varepsilon }u_{\varepsilon }\frac{\partial \psi _{0}}{%
\partial t}dxdt\rightarrow \int_{Q_{T}}u_{0}\frac{\partial \psi _{0}}{%
\partial t}dxdt\text{ as }E^{\prime }\ni \varepsilon \rightarrow 0\text{.}
\end{equation*}%
We are led to
\begin{equation*}
\int_{Q_{T}}\rho ^{\varepsilon }u_{\varepsilon }\frac{\partial \Phi
_{\varepsilon }}{\partial t}dxdt\rightarrow -\int_{0}^{T}\left(
u_{0}^{\prime }(t),\psi _{0}(\cdot ,t)\right) dt.
\end{equation*}%
As far as the second term on the left-hand side of (\ref{5.6}) is concerned,
due to the monotonicity of the function $a(x,t,y,\tau ,\cdot )$ we can argue
as in \cite{CPAA} (see also \cite{AMPA}) to get
\begin{equation*}
\int_{Q_{T}}a^{\varepsilon }(\cdot ,Du_{\varepsilon })\cdot D\Phi
_{\varepsilon }dxdt\rightarrow \iint_{Q_{T}\times Y\times Z}a(\cdot
,Du_{0}+D_{y}u_{1})\cdot (D\psi _{0}+D_{y}\psi _{1})dxdtdyd\tau .
\end{equation*}%
Finally, for the last term, we have
\begin{equation*}
\frac{1}{\varepsilon }\int_{Q_{T}}g^{\varepsilon }(u_{\varepsilon })\Phi
_{\varepsilon }dxdt=\frac{1}{\varepsilon }\int_{Q_{T}}g^{\varepsilon
}(u_{\varepsilon })\psi _{0}dxdt+\int_{Q_{T}}g^{\varepsilon }(u_{\varepsilon
})\psi _{1}^{\varepsilon }dxdt.
\end{equation*}%
It is immediate that
\begin{equation*}
\int_{Q_{T}}g^{\varepsilon }(u_{\varepsilon })\psi _{1}^{\varepsilon
}dxdt\rightarrow \iint_{Q_{T}\times Y\times Z}g(u_{0})\psi _{1}dxdtdyd\tau .
\end{equation*}%
For $\frac{1}{\varepsilon }\int_{Q_{T}}g^{\varepsilon }(u_{\varepsilon
})\psi _{0}dxdt$, we use the decomposition
\begin{equation*}
\frac{1}{\varepsilon }g^{\varepsilon }(u_{\varepsilon })=\Div G^{\varepsilon
}(u_{\varepsilon })-\partial _{r}G^{\varepsilon }(u_{\varepsilon })\cdot
Du_{\varepsilon }
\end{equation*}%
to get
\begin{eqnarray*}
\frac{1}{\varepsilon }\int_{Q_{T}}g^{\varepsilon }(u_{\varepsilon })\Phi
_{\varepsilon }dxdt &=&-\int_{Q_{T}}G^{\varepsilon }(u_{\varepsilon })\cdot
D\psi _{0}dxdt-\int_{Q_{T}}(\partial _{r}G^{\varepsilon }(u_{\varepsilon
})\cdot Du_{\varepsilon })\psi _{0}dxdt \\
&=&I_{1}+I_{2}.
\end{eqnarray*}%
We infer from Lemma \ref{l4.4} that
\begin{equation*}
I_{1}\rightarrow -\iint_{Q_{T}\times Y\times Z}G(u_{0})\cdot D\psi
_{0}dxdtdyd\tau .
\end{equation*}%
Since the function $\partial _{r}G$ is Lipschitz continuous with respect to $%
r$ and periodic with respect to $y,\tau $, the use of Remark \ref{r4.1}
yields
\begin{equation*}
I_{2}\rightarrow -\iint_{Q_{T}\times Y\times Z}(\partial _{r}G(u_{0})\cdot
(Du_{0}+D_{y}u_{1}))\psi _{0}dxdtdyd\tau ;
\end{equation*}%
indeed, this can be verified by using the definition of the strong two-scale
convergence \cite{Allaire, Zhikov}, noting that in our case, the sequence $%
\partial _{r}G^{\varepsilon }(u_{\varepsilon })$ strongly two-scale
converges towards $\partial _{r}G(u_{0})$.

Putting together all the above facts we are led to (\ref{5.3}).\bigskip

\textit{Step 2:} \textbf{Case where} $k=2$. In this case the procedure is
the same as in the previous one. Thus, as it can be seen from the proof of
the case $0<k<2$, we will only deal with the term $\int_{Q_{T}}\rho
^{\varepsilon }u_{\varepsilon }\frac{\partial \Phi _{\varepsilon }}{\partial
t}dxdt$. However, another important fact is to check that $u_{1}$ belongs to
$L^{p}(Q_{T};\mathcal{W})$. This last part will be accomplished in the next
subsection. Returning to (\ref{5.6}) and considering the first term there,
we pass to the limit in the equality
\begin{equation*}
\int_{Q_{T}}\rho ^{\varepsilon }u_{\varepsilon }\frac{\partial \Phi
_{\varepsilon }}{\partial t}dxdt=\int_{Q_{T}}\rho ^{\varepsilon
}u_{\varepsilon }\frac{\partial \psi _{0}}{\partial t}dxdt+\int_{Q_{T}}%
\varepsilon \rho ^{\varepsilon }u_{\varepsilon }\frac{\partial \psi
_{1}^{\varepsilon }}{\partial t}dxdt
\end{equation*}%
using Corollary \ref{c4.1} and we get
\begin{equation*}
\lim_{E^{\prime }\ni \varepsilon \rightarrow 0}\int_{Q_{T}}\rho
^{\varepsilon }u_{\varepsilon }\frac{\partial \Phi _{\varepsilon }}{\partial
t}dxdt=-\int_{0}^{T}\left( u_{0}^{\prime }(t),\psi _{0}(\cdot ,t)\right)
dt-\int_{Q_{T}}\left[ \rho \frac{\partial u_{1}}{\partial \tau },\psi _{1}%
\right] dxdt,
\end{equation*}%
and we hence derive (\ref{5.4}).\bigskip

\textit{Step 3:} \textbf{Case where} $k>2$. As in the preceding step, we
only need to compute the limit (as $E^{\prime }\ni \varepsilon \rightarrow 0$%
) of the term $\int_{Q_{T}}\rho ^{\varepsilon }u_{\varepsilon }\frac{%
\partial \Phi _{\varepsilon }}{\partial t}dxdt$. Before we can do this, we
must first show that the corrector term $u_{1}$ does not depend on $\tau $.
This will allow us to take the test functions independent of $\tau $, that
is, $\psi _{1}\in \mathcal{C}_{0}^{\infty }(Q_{T})\otimes \lbrack \mathcal{D}%
_{\#\rho }(Y)]$, i.e., $\Phi _{\varepsilon }(x,t)=\psi _{0}(x,t)+\varepsilon
\psi _{1}(x,t,x/\varepsilon )$. This will therefore lead at once to
\begin{equation*}
\int_{Q_{T}}\rho ^{\varepsilon }u_{\varepsilon }\frac{\partial \Phi
_{\varepsilon }}{\partial t}dxdt\rightarrow -\int_{0}^{T}\left(
u_{0}^{\prime }(t),\psi _{0}(\cdot ,t)\right) dt\text{ as }E^{\prime }\ni
\varepsilon \rightarrow 0\text{.}
\end{equation*}%
So, let
\begin{equation*}
\psi _{\varepsilon }(x,t)=\varepsilon ^{k-1}\psi \left( x,t,\frac{x}{%
\varepsilon },\frac{t}{\varepsilon ^{k}}\right) \text{, }(x,t)\in Q_{T}\text{%
,}
\end{equation*}%
where $\psi (x,t,y,\tau )=\varphi (x,t)\theta (y)\chi (\tau )$ with $\varphi
\in \mathcal{C}_{0}^{\infty }(Q_{T})$, $\theta \in \mathcal{D}_{\#\rho }(Y)$
and $\chi \in \mathcal{D}_{\text{per}}(Z)$. Then $\psi _{\varepsilon }\in
\mathcal{C}_{0}^{\infty }(Q_{T})$ and as in (\ref{5.6}) we have
\begin{equation}
-\int_{Q_{T}}\rho ^{\varepsilon }u_{\varepsilon }\frac{\partial \psi
_{\varepsilon }}{\partial t}dxdt+\int_{Q_{T}}a^{\varepsilon }(\cdot
,Du_{\varepsilon })\cdot D\psi _{\varepsilon }dxdt-\frac{1}{\varepsilon }%
\int_{Q_{T}}g^{\varepsilon }(u_{\varepsilon })\psi _{\varepsilon }dxdt=0.
\label{5.7}
\end{equation}%
Because $k>2$, the last two terms in the left-hand side of (\ref{5.7}) go to
zero as $E^{\prime }\ni \varepsilon \rightarrow 0$. For the first one we
have
\begin{equation*}
\int_{Q_{T}}\rho ^{\varepsilon }u_{\varepsilon }\frac{\partial \psi
_{\varepsilon }}{\partial t}dxdt=\varepsilon ^{k}\int_{Q_{T}}\frac{1}{%
\varepsilon }\rho ^{\varepsilon }u_{\varepsilon }\left( \frac{\partial \psi
}{\partial t}\right) ^{\varepsilon }dxdt+\int_{Q_{T}}\frac{1}{\varepsilon }%
\rho ^{\varepsilon }u_{\varepsilon }\left( \frac{\partial \psi }{\partial
\tau }\right) ^{\varepsilon }dxdt.
\end{equation*}%
Passing to the limit in the above equation using Lemma \ref{l4.3} gives in (%
\ref{5.7})
\begin{equation*}
\iint_{Q_{T}\times Y\times Z}\rho u_{1}\frac{\partial \psi }{\partial \tau }%
dxdtdyd\tau =0,
\end{equation*}%
and using the arbitrariness of $\varphi $ and $\theta $, we get
\begin{equation*}
\int_{0}^{1}\rho (y)u_{1}(x,t,y,\tau )\frac{\partial \chi }{\partial \tau }%
(\tau )d\tau =0,
\end{equation*}%
which is equivalent to $u_{1}$ is independent of $\tau $. This ends the
proof of Step 3. We are partially done (since we need to check that $u_{1}$,
in the case $k=2$, lies in $L^{p}(Q_{T};\mathcal{W})$).
\end{proof}

\subsection{Homogenized equation}

In this section we consider each of the Eq. (\ref{5.3})-(\ref{5.5})
separately. Let us first and foremost deal with (\ref{5.3}).

Equation (\ref{5.3}) is equivalent to the following system:

\begin{equation}
\left\{
\begin{array}{l}
\iint_{Q_{T}\times Y\times Z}a(\cdot ,Du_{0}+D_{y}u_{1})\cdot
D_{y}v_{1}dxdtdyd\tau \\
=\iint_{Q_{T}\times Y\times Z}g(y,\tau ,u_{0})v_{1}dxdtdyd\tau \text{\ for
all }v_{1}\in \mathcal{C}_{0}^{\infty }(Q_{T})\otimes \mathcal{E}%
\end{array}%
\right.  \label{5.8}
\end{equation}%
and
\begin{equation}
\left\{
\begin{array}{l}
\int_{0}^{T}\left( u_{0}^{\prime },v_{0}\right) dxdt+\iint_{Q_{T}\times
Y\times Z}a(\cdot ,Du_{0}+D_{y}u_{1})\cdot Dv_{0}dxdtdyd\tau \\
+\iint_{Q_{T}\times Y\times Z}G(y,\tau ,u_{0})\cdot Dv_{0}dxdtdyd\tau \\
+\iint_{Q_{T}\times Y\times Z}\left( \partial _{r}G(y,\tau ,u_{0})\cdot
(Du_{0}+D_{y}u_{1})\right) v_{0}dxdtdyd\tau =0 \\
\text{for all }v_{0}\in \mathcal{C}_{0}^{\infty }(Q_{T})\text{.}%
\end{array}%
\right.  \label{5.9}
\end{equation}%
As far as (\ref{5.8}) is concerned, let $(x,t)\in Q_{T}$ and let $(r,\xi
)\in \mathbb{R}\times \mathbb{R}^{N}$ be freely fixed. Let $\pi (x,t,r,\xi )$
be defined by the so-called cell problem
\begin{equation}
\left\{
\begin{array}{l}
\pi (x,t,r,\xi )\in \mathcal{V}=L_{\text{per}}^{p}(Z;W_{\#\rho }^{1,p}(Y)):
\\
\int_{Y\times Z}a(\cdot ,\xi +D_{y}\pi (x,t,r,\xi ))\cdot D_{y}wdyd\tau
=\int_{Y\times Z}g(y,\tau ,r)wdyd\tau \\
\text{for all }w\in \mathcal{V}.%
\end{array}%
\right.  \label{5.10}
\end{equation}%
Since $g(y,\tau ,r)=\Div_{y}G(y,\tau ,r)$, we have
\begin{equation*}
\int_{Y\times Z}g(y,\tau ,r)wdyd\tau =-\int_{Y\times Z}G(y,\tau ,r)\cdot
D_{y}wdyd\tau ,
\end{equation*}%
from which we deduce that the right-hand side of (\ref{5.10}) is a
continuous linear functional on $\mathcal{V}$. It therefore follows from
classical results that Eq. (\ref{5.10}) admits at least a solution. Moreover
if $\pi _{1}\equiv \pi _{1}(x,t,r,\xi )$ and $\pi _{2}\equiv \pi
_{2}(x,t,r,\xi )$ are two solutions of (\ref{5.10}), then we must have
\begin{equation*}
\int_{Y\times Z}\left( a(\cdot ,r,\xi +D_{y}\pi _{1})-a(\cdot ,r,\xi
+D_{y}\pi _{2})\right) \cdot \left( D_{y}\pi _{1}-D_{y}\pi _{2}\right)
dyd\tau =0,
\end{equation*}%
and so, by [part (i) of] (\ref{3.3}),$\;D_{y}\pi _{1}=D_{y}\pi _{2}$, which
means that $\pi _{1}-\pi _{2}$ is a constant function of $y$. \ But then by
the condition $M_{y}(\rho \pi _{1})=M_{y}(\rho \pi _{2})=0$ (recall that $%
\pi _{1}$ and $\pi _{2}$ are in $\mathcal{V}=L_{\text{per}}^{p}(Z;W_{\#\rho
}^{1,p}(Y))$) we deduce that $\pi _{1}=\pi _{2}$. Next, taking in particular
$r=u_{0}(x,t)$ and $\xi =Du_{0}(x,t)$ with $(x,t)$ arbitrarily chosen in $%
Q_{T}$, and then choosing in (\ref{5.8}) the particular test functions $%
v_{1}(x,t)=\varphi (x,t)w$ ($(x,t)\in Q_{T}$) with $\varphi \in \mathcal{C}%
_{0}^{\infty }(Q_{T})$ and $w\in \mathcal{E}$, and finally comparing the
resulting equation with (\ref{5.10}) (note that $\mathcal{E}$ is dense in $%
\mathcal{V}$), the uniqueness of the solution to (\ref{5.10}) tells us that $%
u_{1}=\pi (\cdot ,u_{0},Du_{0})$, where the right-hand side of the preceding
equality stands for the function $(x,t)\mapsto \pi
(x,t,u_{0}(x,t),Du_{0}(x,t))$ from $Q_{T}$ into $\mathcal{V}$.

We have just proved the

\begin{proposition}
\label{p5.2}The solution of the variational problem \emph{(\ref{5.8})} is
unique.
\end{proposition}

Let us now deal with the variational problem (\ref{5.9}). For that, set
\begin{equation*}
q(x,t,r,\xi )=\int_{Y\times Z}a(x,t,\cdot ,\cdot ,\xi +D_{y}\pi (x,t,r,\xi
))dyd\tau \text{ }
\end{equation*}%
and
\begin{equation*}
q_{0}(x,t,r,\xi )=\int_{Y\times Z}\partial _{r}g(y,\tau ,r)\pi (x,t,r,\xi
)dyd\tau
\end{equation*}%
for $(x,t)\in Q_{T}$ and $(r,\xi )\in \mathbb{R}\times \mathbb{R}^{N}$
arbitrarily fixed. With this in mind, we have following

\begin{proposition}
\label{p5.3}The solution $u_{0}$ to the variational problem \emph{(\ref{5.9})%
} solves the following boundary value problem:
\begin{equation}
\left\{
\begin{array}{l}
\frac{\partial u_{0}}{\partial t}=\Div q(\cdot ,\cdot
,u_{0},Du_{0})+q_{0}(\cdot ,\cdot ,u_{0},Du_{0})\text{\ in }Q_{T} \\
u_{0}=0\text{\ on }\partial Q\times (0,T) \\
u_{0}(x,0)=u^{0}(x)\text{\ in }Q\text{.}%
\end{array}%
\right.  \label{5.11}
\end{equation}%
Moreover any subsequential limit in $L^{2}(Q_{T})$ of the sequence $%
(u_{\varepsilon })_{\varepsilon >0}$ is solution to \emph{(\ref{5.11})}.
\end{proposition}

\begin{proof}
Substituting $u_{1}=\pi (\cdot ,u_{0},Du_{0})$ in (\ref{5.9}) and using the
obvious equalities
\begin{equation*}
-\iint_{Q_{T}\times Y\times Z}G(y,\tau ,u_{0})\cdot Dv_{0}dxdtdyd\tau
=\iint_{Q_{T}\times Y\times Z}(\partial _{r}G(y,\tau ,u_{0})\cdot
Du_{0})v_{0}dxdtdyd\tau ,
\end{equation*}%
\begin{equation*}
-\iint_{Q_{T}\times Y\times Z}(\partial _{r}G(y,\tau ,u_{0})\cdot
D_{y}u_{1})v_{0}dxdtdyd\tau =\iint_{Q_{T}\times Y\times Z}\partial
_{r}g(y,\tau ,u_{0})u_{1}v_{0}dxdtdyd\tau ,
\end{equation*}%
Eq. (\ref{5.9}) becomes
\begin{equation*}
\left\{
\begin{array}{l}
\int_{0}^{T}\left( u_{0}^{\prime },v_{0}\right) dxdt+\iint_{Q_{T}\times
Y\times Z}a(\cdot ,Du_{0}+D_{y}\pi (\cdot ,u_{0},Du_{0}))\cdot
Dv_{0}dxdtdyd\tau \\
=\iint_{Q_{T}\times Y\times Z}\partial _{r}g(y,\tau ,u_{0})\pi (\cdot
,u_{0},Du_{0})v_{0}dxdtdyd\tau \text{ for all }v_{0}\in \mathcal{C}%
_{0}^{\infty }(Q_{T})\text{,}%
\end{array}%
\right.
\end{equation*}%
which is nothing else, but the variational formulation of (\ref{5.11}).
\end{proof}

To conclude the study in the case when $0<k<2$, we have the following

\begin{theorem}
\label{t5.2}Let $2\leq p<\infty $. Assume hypotheses \textbf{A1}-\textbf{A5}
hold. For each $\varepsilon >0$ let $u_{\varepsilon }$ be the unique
solution to \emph{(\ref{1.1})}. Then there exists a subsequence of $%
\varepsilon $ not relabeled such that $u_{\varepsilon }\rightarrow u_{0}$ in
$L^{2}(Q_{T})$ where $u_{0}\in L^{p}(0,T;W_{0}^{1,p}(Q))$ is solution to
\emph{(\ref{5.11})}.
\end{theorem}

The case when $k>2$ is quite similar to that when $0<k<2$. Now, let us
consider the case where $k=2$. In that case, all we need to check is that
the solution $u_{1}$ of the microscopic problem is unique and belongs to $%
L^{p}(Q_{T};\mathcal{W})$ as announced in Theorem \ref{t5.1}. For that
purpose, we begin by checking that $u_{1}$ is the solution to the following
variational problem:
\begin{equation}
\left\{
\begin{array}{l}
\int_{Q_{T}}\left[ \rho \frac{\partial u_{1}}{\partial \tau },v_{1}\right]
dxdt+\iint_{Q_{T}\times Y\times Z}a(\cdot ,Du_{0}+D_{y}u_{1})\cdot
D_{y}v_{1}dxdtdyd\tau \\
=\iint_{Q_{T}\times Y\times Z}g(y,\tau ,u_{0})v_{1}dxdtdyd\tau \text{\ for
all }v_{1}\in \mathcal{C}_{0}^{\infty }(Q_{T})\otimes \mathcal{E}\text{.}%
\end{array}%
\right.  \label{5.12}
\end{equation}%
Fix $(r,\xi )\in \mathbb{R}\times \mathbb{R}^{N}$ and $(x,t)\in Q_{T}$, and
consider the cell problem
\begin{equation}
\left\{
\begin{array}{l}
\pi \equiv \pi (x,t,r,\xi )\in \mathcal{V}=L_{\text{per}}^{p}(Z;W_{\#\rho
}^{1,p}(Y)) \\
\left[ \rho \frac{\partial \pi }{\partial \tau },w\right] +\int_{Y\times
Z}a(\cdot ,\xi +D_{y}\pi )\cdot D_{y}wdyd\tau =\int_{Y\times Z}g(y,\tau
,r)wdyd\tau \\
\text{for all }w\in \mathcal{E}=\mathcal{D}_{\text{per}}(Z)\otimes \mathcal{D%
}_{\#\rho }(Y)\text{.}%
\end{array}%
\right.  \label{5.13}
\end{equation}%
Assume for a while that the solution of (\ref{5.13}) exists. Then, for the
same reasons as in the case where $0<k<2$, the linear functional $L:w\mapsto
\int_{Y\times Z}g(y,\tau ,r)wdyd\tau $ defined on $\mathcal{V}$ verifies the
property: there is a positive constant $c$ independent of $w$ such that
\begin{equation*}
\left\vert L(w)\right\vert \leq c\left\Vert w\right\Vert _{\mathcal{V}}\text{
for all }w\in \mathcal{V}.
\end{equation*}%
Likewise there exists another constant $k>0$ such that
\begin{equation*}
\left\vert \int_{Y\times Z}a(\cdot ,\xi +D_{y}\pi )\cdot D_{y}wdyd\tau
\right\vert \leq k\left\Vert w\right\Vert _{\mathcal{V}}.
\end{equation*}%
We deduce from the above facts that the linear functional $w\mapsto \left[
\rho \frac{\partial \pi }{\partial \tau },w\right] $, defined on $\mathcal{E}
$, is continuous when endowing $\mathcal{E}$ with the $\mathcal{V}$-norm.
From the density of $\mathcal{E}$ in $\mathcal{V}$ we get readily $\rho
\frac{\partial \pi }{\partial \tau }\in \mathcal{V}^{\prime }$, so that $\pi
$ lies in $\mathcal{W}$. Since $\mathcal{E}$ in $\mathcal{W}$, Eq. (\ref%
{5.13}) still holds for $w\in \mathcal{W}$. Therefore by equality (\ref{4.2}%
) (in Proposition \ref{p4.1}) we deduce that $\left[ \rho \frac{\partial \pi
}{\partial \tau },\pi \right] =0$. The uniqueness of the solution of (\ref%
{5.13}) follows from that. So it remains to show that Eq. (\ref{5.13})
possesses at least a solution. But this equation is the variational
formulation of the following equation:
\begin{equation}
\left\{
\begin{array}{l}
\pi \in \mathcal{W}: \\
\rho \frac{\partial \pi }{\partial \tau }=\Div_{y}a(\cdot ,\xi +D_{y}\pi
)+g(\cdot ,\cdot ,r).%
\end{array}%
\right.  \label{5.14}
\end{equation}%
In view of the properties of the operator $\mathcal{R}$ defined in Section
4, we see immediately by \cite{Paronetto} (see also \cite{Alt}) that the
above equation admits at least a solution. Now, taking $r=u_{0}(x,t)$ and $%
\xi =Du_{0}(x,t)$ and arguing as in the case where $0<k<2$, we obtain $%
u_{1}=\pi (\cdot ,\cdot ,u_{0},Du_{0})$. By the preceding equality, we have
shown, as claimed, that $u_{1}$ lies in $L^{p}(Q_{T};\mathcal{W})$, thereby
concluding the proof of Theorem \ref{t5.1}. This also shows that even in
this case, the homogenized equation still has the form (\ref{5.11}).

\subsection{Some uniqueness results and convergence of the sequence $(u_{%
\protect\varepsilon })_{\protect\varepsilon >0}$}

In order to find a uniqueness result for the solution of the problem (\ref%
{5.11}), we need to know the properties of the homogenized coefficients. The
properties of the function $q$ are classically known. However it is
difficult to have the general properties of the function $q_{0}$. But we
will nevertheless show that in some cases, there is uniqueness. For this, we
will restrict the study to a special case: we assume that the function $%
\lambda \mapsto a(x,t,y,\tau ,\lambda )$, from $\mathbb{R}^{N}$ into itself
is linear, that is, there exists a family $\{a_{ij}\}_{1\leq i,j\leq
N}\subset \mathcal{C}(\overline{Q}_{T};L^{\infty }(\mathbb{R}_{y,\tau
}^{N+1}))$ (thanks to (\ref{3.1}) and parts (ii) and (iii) of (\ref{3.3})),
such that
\begin{equation*}
a_{i}(x,t,y,\tau ,\lambda )=\sum_{j=1}^{N}a_{ij}(x,t,y,\tau )\lambda _{j}%
\text{ for all }\lambda \in \mathbb{R}^{N}\text{ (}1\leq i\leq N\text{).}
\end{equation*}%
In the sequel, we assume $p=2$. It is clear that the results obtained in the
preceding sections are still valid in this case. From the periodicity
assumption on $a(x,t,\cdot ,\cdot ,\lambda )$, it is clear that the
functions $a_{ij}(x,t,\cdot ,\cdot )$ are $Y\times Z$-periodic.

Set $b=(a_{ij})_{1\leq i,j\leq N}$ (the matrix derived from the coefficients
$a_{ij}$) and let us focused our attention on the special case where $k=2$,
which seems to be the more involved. The cell problem (\ref{5.14}) takes the
form
\begin{equation}
\left\{
\begin{array}{l}
\pi \equiv \pi (x,t,r,\xi )\in \mathcal{W}: \\
\rho \frac{\partial \pi }{\partial \tau }=\Div_{y}\left( b(x,t,\cdot ,\cdot
)(\xi +D_{y}\pi )\right) +g(\cdot ,\cdot ,r).%
\end{array}%
\right.  \label{5.15}
\end{equation}%
We know that this equation has a unique solution. But it can be easily seen
that the solution of the above equation expresses under the form
\begin{equation}
\pi (x,t,r,\xi )(y,\tau )=\chi (x,t,y,\tau )\cdot \xi +w_{1}(x,t,y,\tau ,r)
\label{5.16}
\end{equation}%
where $\chi $ and $w_{1}$ are respective unique solutions to the following
equations

\begin{equation*}
\rho \frac{\partial \chi }{\partial \tau }-\Div_{y}(b(x,t)D_{y}\chi )=\Div%
_{y}b\text{ in }\mathcal{W}^{\prime },\;\chi =\chi (x,t,\cdot ,\cdot )\in (%
\mathcal{W})^{N},
\end{equation*}%
and
\begin{equation*}
\rho \frac{\partial w_{1}}{\partial \tau }-\Div_{y}(b(x,t)D_{y}w_{1})=g(%
\cdot ,\cdot ,r)\text{ in }\mathcal{W}^{\prime },\;w_{1}=w_{1}(x,t,\cdot
,\cdot ,r)\in \mathcal{W}
\end{equation*}%
where $b(x,t)$ stands for the matrix $(a_{ij}(x,t,\cdot ,\cdot ))_{1\leq
i,j\leq N}$. The existence and uniqueness of $\chi $ and $w_{1}$ is ensured
by a classical result \cite{Paronetto}.

Now, taking $r=u_{0}(x,t)$ and $\xi =Du_{0}(x,t)$ in (\ref{5.15}), it
follows from (\ref{5.16}) that
\begin{equation*}
u_{1}(x,t,y,\tau )=\chi (x,t,y,\tau )\cdot Du_{0}(x,t)+w_{1}(x,t,y,\tau
,u_{0}(x,t)).
\end{equation*}%
Now, going back to the variational formulation of (\ref{5.4}) with the
function $u_{1}$ replaced by the above expression, we end up with
\begin{equation*}
\left\{
\begin{array}{l}
\int_{0}^{T}\left( u_{0}^{\prime },v_{0}\right) dxdt+\int_{Q_{T}}(\widehat{b}%
(x,t)Du_{0})\cdot Dv_{0}dxdt \\
+\iint_{Q_{T}\times Y\times Z}b(x,t)D_{y}w_{1}(x,t,u_{0})\cdot
Dv_{0}dxdtdyd\tau \\
=\iint_{Q_{T}\times Y\times Z}\partial _{r}g(y,\tau ,u_{0})(\chi (x,t,y,\tau
)\cdot Du_{0}(x,t))dxdtdyd\tau \\
+\iint_{Q_{T}\times Y\times Z}\partial _{r}g(y,\tau ,u_{0})w_{1}(x,t,y,\tau
,u_{0}(x,t))v_{0}dxdtdyd\tau \text{ for all }v_{0}\in \mathcal{C}%
_{0}^{\infty }(Q_{T})\text{,}%
\end{array}%
\right.
\end{equation*}%
where $\widehat{b}(x,t)=\int_{Y\times Z}b(x,t)(I+D_{y}\chi )dyd\tau $ is the
homogenized matrix, $I$ being denoting the unit $N\times N$ matrix. Setting
\begin{eqnarray*}
F_{1}(x,t,r) &=&\int_{Y\times Z}b(x,t)D_{y}w_{1}(x,t,y,\tau ,r))dyd\tau
\text{;} \\
F_{2}(x,t,r) &=&\int_{Y\times Z}\partial _{r}g(y,\tau ,r)\chi (x,t,y,\tau
)dyd\tau \text{;} \\
F_{3}(x,t,r) &=&\int_{Y\times Z}\partial _{r}g(y,\tau ,r)w_{1}(x,t,y,\tau
,r)dyd\tau \text{,}
\end{eqnarray*}%
we are led to the following result.

\begin{proposition}
\label{p5.4}The solution $u_{0}$ to the variational problem \emph{(\ref{5.4})%
} solves the following boundary value problem:
\begin{equation}
\left\{
\begin{array}{l}
\frac{\partial u_{0}}{\partial t}=\Div\left( \widehat{b}(x,t))Du_{0}\right) +%
\Div F_{1}(x,t,u_{0})-F_{2}(x,t,u_{0})\cdot Du_{0}-F_{3}(x,t,u_{0})\text{\
in }Q_{T} \\
u_{0}=0\text{\ \ on }\partial Q\times (0,T) \\
u_{0}(x,0)=u^{0}(x)\text{\ \ in }Q\text{.}%
\end{array}%
\right.  \label{5.17}
\end{equation}
\end{proposition}

As in \cite{AllPiat1}, it can be checked straightforwardly that the
functions $F_{i}(x,t,\cdot )$ ($1\leq i\leq 3$) are Lipschitz continuous
functions. This therefore ensures the uniqueness of the solution to (\ref%
{5.17}), and the following result holds true.

\begin{theorem}
\label{t5.3}Assume hypotheses \textbf{A1}-\textbf{A5} hold with $p=2$. For
each $\varepsilon >0$ let $u_{\varepsilon }$ be the unique solution to \emph{%
(\ref{1.1})}. Then $u_{\varepsilon }\rightarrow u_{0}$ in $L^{2}(Q_{T})$ as $%
\varepsilon \rightarrow 0$, where $u_{0}\in L^{2}(0,T;H_{0}^{1}(Q))$ is the
unique solution to \emph{(\ref{5.17})}.
\end{theorem}

\begin{proof}
By the uniqueness of the solution to (\ref{5.17}), the result follows in an
obvious way.
\end{proof}

The same remark as above holds in all the other cases (as far as the
parameter $k$ is concerned), so that we are justified in saying that Theorem %
\ref{t5.3} holds for any positive value of the parameter $k$. This shows the
convergence of the sequence $(u_{\varepsilon })_{\varepsilon >0}$ when the
function $a(x,t,y,\tau ,\lambda )$ is linear with respect to $\lambda $.
Also we recover the results by Allaire and Piatnitski \cite{AllPiat1} (when
setting in our situation $k=2$) when the function $a(x,t,y,\tau ,\lambda )$
is linear and does not depend on the variable $x,t$. We can therefore argue
that our work generalize the one of the previous authors.

\begin{acknowledgement}
\emph{The research of the second author was partially supported by the
University of Pretoria and the National Research Foundation of South Africa.}
\end{acknowledgement}


\begin{thebibliography}{99}
\bibitem{Allaire} G. Allaire, Homogenization and two-scale convergence, SIAM
J. Math. Anal\textit{.} \textbf{23} (1992) 1482--1518.

\bibitem{AllPiat1} G. Allaire, A. Piatnitski, Homogenization of nonlinear
reaction-diffusion equation with a large reaction term, Ann. Univ. Ferrara
\textbf{56} (2010) 141-161.

\bibitem{Alt} H.W. Alt, S. Luckhaus, Quasilinear elliptic-parabolic
differential equations, Math. Z. \textbf{183} (1983) 311-341.

\bibitem{Amar} M. Amar, A. Dall'Aglio, F. Paronetto, Homogenization of
forward-backward parabolic equations, Asymptotic Anal. \textbf{42} (2005)
123-132.

\bibitem{Amaziane} B. Amaziane, L. Pankratov, A. Piatnitski, Nonlinear flow
through double porosity media in variable exponent Sobolev spaces, Nonlinear
Anal. RWA \textbf{10} (2009) 2521-2530.

\bibitem{Antontsev} S.N. Antontsev, S.I. Shamaev, A model porous medium
equation with variable exponent of nonlinearity: Existence, uniqueness and
localization properties of solutions, Nonlinear Anal. TMA \textbf{60} (2005)
515-545.

\bibitem{BLP} A. Bensoussan, J.L. Lions, G. Papanicolaou, Asymptotic
analysis for periodic structures, North Holland, Amsterdam, 1978.

\bibitem{DiopPardoux} M.A. Diop, B. Iftimie, E. Pardoux, A.L. Piatnitski,
Singular homogenization with stationary in time and periodic in space
coefficients, J. Functional Anal. \textbf{231} (2006) 1-46.

\bibitem{Jikov} V.V. Jikov, S.M. Kozlov, O.A. Oleinik, Homogenization of
differential operators and integral functionals, Springer-Verlag, Berlin,
1994.

\bibitem{Lions} J.\ L.\ Lions, Quelques m\'{e}thodes de r\'{e}solution des
probl\`{e}mes aux limites non lin\'{e}aires, Dunod, Paris, 1969.

\bibitem{Nguetseng} G. Nguetseng, A general convergence result for a
functional related to the theory of homogenization, SIAM J. Math. Anal%
\textit{.} \textbf{20} (1989) 608-623.

\bibitem{EJDE} G. Nguetseng, J.L. Woukeng, Deterministic homogenization of
parabolic monotone operators with time dependent coefficients, Electron. J.
Differ. Equ. \textbf{2004} (2004) 1-23.

\bibitem{NgWou} G. Nguetseng, J. L. Woukeng, $\Sigma $-convergence of
nonlinear parabolic operators, Nonlinear Anal. TMA \textbf{66} (2007)
968-1004.

\bibitem{Pankov} A. Pankov, T.E. Pankova, Nonlinear evolution equations with
non-invertible operator coefficient at the derivative, Dokl. Akad. Nauk
Ukrainy \textbf{9} (1993) 18-20 (in Russian).

\bibitem{PardouxPiat} E. Pardoux, A.L. Piatnitski, Homogenization of a
nonlinear random parabolic partial differential equation, Stochast.
Processes Appl. \textbf{104} (2003) 1-27.

\bibitem{Paronetto} F. Paronetto, Homogenization of degenerate
elliptic-parabolic equations, Asymptotic Anal. \textbf{37} (2004) 21-56.

\bibitem{Paronetto1} F. Paronetto, Homogenization of a class of degenerate
parabolic equations, Asymptotic Anal. \textbf{21} (1999) 275-302.

\bibitem{Paronetto2} F. Paronetto, Some new results on the convergence of
degenerate elliptic and parabolic equations, J. Convex Anal. \textbf{9}
(2002) 31-54.

\bibitem{Paronetto3} F. Paronetto, F. Serra Cassano, On the convergence of
class of degenerate parabolic equations, J. Math. Pures Appl\textit{.}
\textbf{77} (1998) 735-759.

\bibitem{AMPA} J.L. Woukeng, Periodic homogenization of nonlinear
non-monotone parabolic operators with three time scales, Ann. Mat. Pura
Appl. \textbf{189} (2010) 357-379.

\bibitem{CPAA} J.L. Woukeng, $\Sigma $-convergence and reiterated
homogenization of nonlinear parabolic operators, Commun. Pure Appl. Anal.
\textbf{9} (2010) 1753-1789.

\bibitem{Zeidler} E. Zeidler, Nonlinear Functional Analysis and its
Applications, Vols II A and II B, Springer, New York, 1990.

\bibitem{Zhikov} V.V. Zikhov, On the two-scale convergence, J. Math. Sci.
\textbf{120} (2004) 1328-1352.
\end{thebibliography}
\end{document}